\documentclass
{amsart}

\usepackage{multirow}
\usepackage[square, numbers, comma, sort&compress]{natbib} 
\newcommand{\lvt}{\left|\kern-1.35pt\left|\kern-1.3pt\left|}
\newcommand{\rvt}{\right|\kern-1.3pt\right|\kern-1.35pt\right|}

\usepackage[foot]{amsaddr}
\usepackage[english]{babel}
\usepackage[pdftex,hyperfootnotes]{hyperref}
\usepackage[usenames,dvipsnames,svgnames,table,x11names]{xcolor}
\usepackage{pgfplots}
\usepackage{tikz}
\usepackage{tikz-3dplot}
\usetikzlibrary{calc,shadows,shapes.callouts,shapes.geometric,shapes.misc,positioning,patterns,decorations.pathmorphing,decorations.markings,decorations.fractals,decorations.pathreplacing,shadings,fadings}

\usepackage{longtable}
\usepackage{enumerate}
\allowdisplaybreaks

\usepackage{fancybox}
\usepackage[utf8]{inputenc}

\usepackage{rotating,multirow,pdflscape}
\usepackage{amssymb,latexsym,amsmath,amsthm}
\usepackage{mathrsfs}
\usepackage{mathtools,arydshln,mathdots}
\usepackage{bigints}
\usepackage{comment}

\makeatletter
\renewcommand*\env@matrix[1][*\c@MaxMatrixCols c]{%
 \hskip -\arraycolsep
 \let\@ifnextchar\new@ifnextchar
 \array{#1}}
\makeatother

\mathtoolsset{centercolon}
\usepackage[table]{xcolor}
\usepackage[toc,page]{appendix}

\usepackage{tocvsec2}

\usepackage{Baskervaldx}
\usepackage{newtxmath }

\newtheorem{defi}{Definition}
\newtheorem{teo}{Theorem}
\newtheorem{pro}{Proposition}

\newcommand{\ii}{\hspace{-.045cm}\operatorname{i}}

\renewcommand{\d}{\operatorname{d} \hspace{-.0515cm}}
\newcommand{\Exp}[1]{\operatorname{e}^{#1}}

\renewcommand{\Re}{\operatorname{Re}}

\newcommand{\C}{\mathbb{C}}

\newcommand{\N}{\mathbb{N}}

\newcommand{\0}{{\bf \mathbf{0}}}

\makeatletter
\def\@settitle{\begin{center}
 \baselineskip14\p@\relax
 \bfseries
 \uppercasenonmath\@title
 \@title
 \ifx\@subtitle\@empty\else
 \\[1ex]\uppercasenonmath\@subtitle
 \footnotesize\mdseries\@subtitle
 \fi
 \end{center}
}
\def\subtitle#1{\gdef\@subtitle{#1}}
\def\@subtitle{}
\makeatother

\title[Matrix Bessel Biorthogonal Polynomials: A Riemann--Hilbert approach]{Matrix Bessel Biorthogonal Polynomials:
 \\
A Riemann--Hilbert approach}



\author[Branquinho]{Amílcar Branquinho\(^1\)}
\address{\(^1\)Departamento de Matem\'atica,
Universidade de Coimbra, 3001-454 Coimbra, Portugal}
\email{ajplb@mat.uc.pt}

\author[Foulqui\'e-Moreno]{Ana Foulqui\'e-Moreno\(^2\)}
\address{\(^2\)Departamento de Matemática, Universidade de Aveiro, 3810-193 Aveiro, Portugal}
\email{foulquie@ua.pt}

\author[Fradi]{Assil Fradi\(^3\)}
\address{ \(^3\)Mathematical Physics Special Functions and Applications Laboratory, The Higher School of Sciences and Technology of Hammam Sousse, University of Sousse, Sousse 4002, Tunisia}
\email{assilfradi@ua.pt}

\author[Mañas]{Manuel Mañas\(^4\)}
\address{ \(^4\)Departamento de Física Teórica, Universidad Complutense de Madrid, 28040-Madrid, Spain
}
\email{manuel.manas@ucm.es}

\thanks{ \(^1\)Acknowledges Centre for Mathematics of the University of Coimbra (funded by the Portuguese Government through FCT/MCTES, Multi-Annual Financing Program for R\&D Units, doi: 10.54499/UIDB/00324/2020).}

\thanks{ \(^2\) and \(^3\)Acknowledge the CIDMA Center for Research and Development in Mathematics and Applications (University of Aveiro) and the Portuguese Foundation for Science and Technology (FCT) Multi-Annual Financing Program for R\&D Units, for their support within projects doi: 10.54499/UIDB/04106/2020 \& doi: 10.54499/UIDP/04106/2020}

\thanks{ \(^4\)Acknowledge PID2021-122154NB-I00, entitled ``Ortogonalidad y Aproximaci\'on con Aplicaciones en Machine Learning y Teor\'\i a de la Probabilidad'', funded by MICIU/AEI/10.13039 /501100011033 and by ``RDF A Way of making Europe''.}

\begin{document}

%\begin{center}
%\large
%Matrix Bessel Biorthogonal Polynomials:
% \\
%A Riemann--Hilbert approach
%\end{center}
%
%\medskip

\maketitle

\begin{abstract}
We consider matrix orthogonal polynomials related to Bessel type matrices of weights that can be defined in terms of a given matrix Pearson equation.
From a~Riemann--Hilbert problem we derive first and second order differential relations for the matrix orthogonal polynomials and functions of second kind.
It is shown that the corresponding matrix recurrence coefficients satisfy a non-Abelian extensions of a family of discrete Painlev\'e d-PIV equations.
We present some nontrivial examples of matrix orthogonal polynomials of Bessel type.
\end{abstract}

\section{Introduction}\label{sec:1}

In~\cite{zbMATH03048838}, Krall and Frink introduced the 
%following
generalized Bessel polynomials (GBPs),
\begin{align}\label{eq:scalarbessel}
y_n(x,a,b)=
\,_2F_0\left(-n,a+n-1; - ;-\frac{x}{b}\right)
, && n \in \mathbb{N} \coloneqq \{ 0 ,1 , 2, \ldots \} .
\end{align}
These polynomials satisfy an orthogonality property on the unit circle, 
for all \( m, n \in \mathbb{N} \),
\begin{align}\label{Ortho-unit-circle}
 \int_{|z|=1} \rho^{(a, b)}(z) y_m(z , a, b) y_n(z , a, b) \frac{\d z}{2 \pi \ii} 
= \frac{(-1)^{n+1} n! \, b \, \Gamma(a) \, \delta_{m,n}}{(a+2n-1)\Gamma(a+n-1)}
 , 
\end{align}
%\( m, n \in \mathbb{N} \coloneqq \{ 0 ,1,2, \ldots \} \),
where \( \delta_{m,n}\) is the Kronecker delta, \( \Gamma \) denotes the gamma function, and the weight function \( \rho^{(a,b)}\) is given by
\begin{align}\label{weight_not_Pearson}
\rho^{(a, b)}(z) 
\coloneqq
 \sum_{n=0}^{\infty} \frac{\Gamma(a)}{\Gamma(a+n-1)}\Big(-\frac{b}{z}\Big)^n.
\end{align}
As noted in~\cite{sym15040822}, the integral in~\eqref{Ortho-unit-circle} is not an inner product with the kernel
\begin{align*}
\rho^{(a, b)}(z) y_m(z , a, b) \overline{y_n(z , a, b)} ,
\end{align*}
which is unexpected for polynomials orthogonal on curves in the complex plane.
In the special case where \( a \in \mathbb{N}\) and \( m, n \in \mathbb{N} \), Burchnall provided a simpler orthogonality property over any closed contour \( C \) surrounding the origin \( z = 0\) in the complex plane (see \cite{zbMATH03064328}),
for all \( m, n \in \mathbb{N} \),
\begin{align*}
 \int_{C} z^{a-2} \Exp{-\frac{b}{z}} y_m(z, a, b) y_n(z, a, b) \frac{\d z}{2 \pi \ii}
= \frac{(-1)^{a+n-1} b^{a-1} n!}{(a+n-2)!(a+2n-1)} \delta_{m,n} .
%&& m, n \in \mathbb{N} .
\end{align*}
Later, an orthogonality relation on the positive real axis for all \( m, n \in \mathbb{N} \), with \( \Re(a) < 1 - m - n\),
was given in~\cite{MR913042}, where the scaling factor \( b \) has been put equal to unity without loss of generality,
for all \( m, n \in \mathbb{N} \),
\begin{align*}
\int_0^{\infty} x^{a-2} \Exp{-\frac{b}{x}} y_m(x, a, b) y_n(x, a, b) \d x
= \frac{(-1)^n b^{a-1} n!}{(a+2n-1) \Gamma(a+n-1)} \frac{\pi}{\sin(\pi a)} \delta_{m,n},
%&& m, n \in \mathbb{N} ,
\end{align*}
and in the same work, the following was established,
for all \( m, n \in \mathbb{N} \),
\begin{align*}
 \int_0^{(\infty+)} z^{a-2} \Exp{-\frac{b}{z}} y_m(z; a, b) y_n(z; a, b) \frac{\d z}{2 \pi \ii}
= \frac{(-1)^{a+n} b^{a-1} n!}{(a+2n-1) \Gamma(a+n-1)} \delta_{m,n}, 
%&& m, n \in \mathbb{N} ,
\end{align*}
where the contour is taken along a simple loop starting at the origin \( (z = 0)\), encircling the point at infinity once in the positive (counterclockwise) direction, and then returning to the~origin.
%In this work, we focus on this orthogonality where \( \gamma \) will be that contour. 
For this choice of contour,
\( \sigma(z) = z^{a-2} \Exp{-b / z}\), 
satisfies the Pearson equation
%\begin{align*}
\(
\big(z^2 \sigma\big)^\prime = (az + b)\sigma \),
%\end{align*}
which is crucial for deriving the main explicit formulas in this work.

On the other hand, the weight function \( \rho^{(a,b)}\), defined in~\eqref{weight_not_Pearson} and satisfying the orthogonality relation~\eqref{Ortho-unit-circle}, does not satisfy a Pearson equation except in the cases \( a = 1\) or \( a = 2\). Specifically,
\begin{align*}
\big(z^2 \rho^{(a,b)}\big)^{\prime} = (a z + b) \rho^{(a,b)} - (a-1)(a-2)z .
\end{align*}
For further details, see~\cite[equation 40]{zbMATH03048838} and~\cite{sym15040822} for an overview on this subject.

Our study of matrix valued orthogonal polynomials of Bessel type is motivated by the fact that the systems of orthogonal polynomials associated with the names of Hermite, Laguerre, Jacobi (including the special cases named after Chebyshev, Legendre, and Gegenbauer), and~Bessel are among the most widely studied and applied systems. However, while a Riemann--Hilbert approach has been used to discuss matrix biorthogonal polynomials (MBPs) of Hermite type \cite{zbMATH07310648,zbMATH06055477,zbMATH07258831,zbMATH06050717}, Laguerre type \cite{math10081205}, and Jacobi type \cite{zbMATH07757742}, to our knowledge, this approach has only been applied to scalar Bessel-type orthogonal polynomials \cite{zbMATH06401645}. Therefore, in this work, we focus on matrix valued orthogonal polynomials of Bessel type.

Our aim is to investigate the MBPs associated with Bessel-type matrices of weights, which are constructed in terms of matrix Pearson equations.

In this work, we deal with \emph{regular matrix of weights} \( W(z)\),~i.e., their moments
%\begin{align*}
\(
W_n 
%\coloneqq
= \int_\gamma z^n W (z) \frac{\d z} {2\pi\ii} \),
%&&
\( n\in\N \),
%\end{align*}
are such that
\( \det \big[ W_{j+k} \big]_{j,k=0, \ldots n}\not = 0 \), \( n \in \N \).

\begin{defi}\label{def:1}
We say that a \( N\times N \) \emph{matrix of weights} 
\( W=\begin{bsmallmatrix} 
W^{(1,1)} & \cdots & W^{(1,N)} \\
\vdots & \ddots & \vdots \\
W^{(N,1)} & \cdots & W^{(N,N)} 
\end{bsmallmatrix}\) with support~\( \gamma \),
taken along a simple loop starting at the origin, encircling the point at infinity once in the positive 
%(counterclockwise)
direction, and then returning to the origin,
 is of \emph{Bessel type} if
the entries~\( W^{(j, k)}\) of the matrix measure \( W \) can be written as
\begin{align}\label{eq:the_weights_Bessel}
W^{(j, k)}(z)=\sum_{m \in I_{j, k}} \varphi_{m}(z) z^{p_{m}} \log^{q_{m}} (z) \Exp{-\frac{\lambda_m}{z}}
 , 
&&
z \in \gamma ,
\end{align} 
where \( I_{j, k}\) denotes a finite set of indexes, \( j,k \in \{ 1,2 , \ldots, N \}\), \( \Re(\lambda_m)\,\geq 0\), \( q_m\in \N\) and \( \varphi_{m}\)
is H\"{o}lder con\-tin\-u\-ous,
bounded and nonvanishing on~\( \gamma \).
\end{defi}

We will request, in the development of the theory, that the functions~\( \varphi_{m}\) have a holomorphic extension to the whole complex plane.

Given a regular matrix of weights \( W \), we define two \emph{sequences of matrix monic polynomials},
 \( \big\{ P_{n}^{\mathsf L} (z) \big\}_{n \in \mathbb{Z}_{+}}\) and \( \big\{ P_{n}^{\mathsf R}(z)\big\}_{n \in \mathbb{N}}\), which are left, respectively right-orthogonal. These polynomials satisfy the degree conditions
\( \deg P_{n}^{\mathsf L} (z)=n\), \( \deg P_{n}^{\mathsf R} (z)=n\), \( n \in \mathbb{N}\),
and are defined by
 the~or\-thog\-o\-nal\-i\-ty relations,
\begin{align} \label{eq:ortogonalidadL}
 \int_\gamma {P}_n^{\mathsf L} (z) W (z) z^k \frac{\d z} {2\pi\ii}
= \delta_{n,k} C_n^{-1} , &&
 \int_\gamma z^k W (z) {P}_n^{\mathsf R} (z) \frac{\d z} {2\pi\ii}
= \delta_{n,k} C_n^{-1} , 
\end{align}
for all \( k \in \{ 0, 1, \ldots , n \} \) and \( n \in \mathbb{N}\),
where \( C_n \) is a nonsingular matrix.

The weight matrix \( W \) induces a sesquilinear form on the set of matrix polynomials~\( \C^{N\times N}[z]\) given by
\begin{align*}
\langle P,Q \rangle_W \coloneqq \int_\gamma {P} (z) W (z) Q(z) \frac{\d z} {2\pi\ii},
\end{align*}
with respect to which the sequences \( \big\{ P_n^{\mathsf L}(z)\big\}_{n\in\N}\) and \( \big\{ P_n^{\mathsf R}(z)\big\}_{n\in\N}\) are biorthogonal,
%\begin{align*}
\(
\big\langle P_n^{\mathsf L}, {P}_m^{\mathsf R} \big\rangle_ W
%& 
 = \delta_{n,m} C_n^{-1}
 \), %&
\(n , m \in \mathbb N \).
%\end{align*}
Since the polynomials are monic, they can be expressed as follows,
\begin{align*}
{P}_n^{\mathsf L} (z) &= \operatorname I z^n + p_{\mathsf L,n}^1 z^{n-1} + p_{\mathsf L,n}^2 z^{n-2} + \cdots + p_{\mathsf L,n}^n , \\
{P}_n^{\mathsf R} (z) &= \operatorname I z^n + p_{\mathsf R,n}^1 z^{n-1} + p_{\mathsf R,n}^2 z^{n-2} + \cdots + p_{\mathsf R,n}^n , 
\end{align*}
where \( p_{\mathsf L, n}^k , p_{\mathsf R, n}^k \in 
\C^{N\times N} \), \( k \in \{ 0, \ldots, n \} \) and \( n \in \mathbb N \), with the condition \( p_{\mathsf L,n}^0 = p_{\mathsf R,n}^0=\operatorname I \), \( n \in \mathbb N \) and \( \operatorname I\in\C^{N\times N}\) denotes the identity matrix.

We define the \emph{sequence of second kind matrix functions}, for all \( n \in \mathbb N \), by
\begin{align*}
%\label{eq:secondkind}
Q^{\mathsf L}_n (z)
\coloneqq \int_\gamma \frac{P^{\mathsf L}_n (z')}{z'-z} { W (z')} \frac{\d z^\prime} {2\pi\ii} , 
 &&
{Q}_n^{\mathsf R} (z)
\coloneqq \int_\gamma W (z') \frac{P^{\mathsf R}_{n} (z')}{z'-z} \frac{\d z^\prime} {2\pi\ii} .
% && n \in \mathbb N .
\end{align*} 
From the orthogonality conditions~\eqref{eq:ortogonalidadL}, the following asymptotic expansions hold for all \( n \in \N\), 
as~\( |z|\to\infty\)
\begin{align} \label{eq:secondkindlaurent}
\hspace{-.135cm}
Q^{\mathsf L}_n (z) = - C_n^{-1} z^{-n}\big( \operatorname I z^{-1} + q_{\mathsf L,n}^1 z^{-2} + \cdots \big),
Q^{\mathsf R}_n (z) = - z^{-n} \big( \operatorname I z^{-1} + q_{\mathsf R,n}^1 z^{-2} + \cdots \big)C_n^{-1} . 
\end{align}

The sequences \( \big\{ {P}^{\mathsf L}_n (z)\big\}_{n\in\N} \) and \( \big\{ {Q}^{\mathsf L}_n (z)\big\}_{n\in\N} \) satisfy the following three-term recurrence~relations, for all \( n \in \mathbb{N} \),
\begin{align*}
z {P}^{\mathsf L}_n (z) &= {P}^{\mathsf L}_{n+1} (z) + \xi^{\mathsf L}_n {P}^{\mathsf L}_{n} (z) + \eta^{\mathsf L}_n {P}^{\mathsf L}_{n-1} (z), 
%&& n \in \mathbb N, \\
&&
z Q_{n}^{\mathsf L}(z) = Q_{n+1}^{\mathsf L}(z) + \xi_{n}^{\mathsf L} Q_{n}^{\mathsf L}(z) + \eta^{\mathsf L}_n Q_{n-1}^{\mathsf L}(z),
%&& n \in \mathbb{N},
\end{align*}
where \( \eta^{\mathsf L}_n 
%\coloneqq
 = C_n^{-1} C_{n-1}\), 
with the initial conditions 
%\begin{align*}
\(
{P}^{\mathsf L}_{-1} = \0 \), 
%&& 
\(
{P}^{\mathsf L}_{0} = \operatorname I\),
%&&
\(
 Q_{-1}^{\mathsf L}(z) = -C_{-1}^{-1}\),
 % && \text{
 and
 %} && 
\(
 Q_{0}^{\mathsf L}(z) = S_{W}(z)\),
%\end{align*}
and \( S_{W}(z)\) is the Stieltjes--Markov transform given by 
\begin{align*}
S_{W}(z) \coloneqq \int_{\gamma} \frac{W\left(z^{\prime}\right)}{z^{\prime}-z} 
\frac{\d z^\prime}{2 \pi \ii }.
\end{align*}
Similarly, the sequences of monic polynomials \( \big\{ {P}^{\mathsf R}_n (z)\big\}_{n\in\N} \) and second kind matrix functions \( \big\{ {Q}^{\mathsf R}_n (z)\big\}_{n\in\N} \) satisfy the relations, for all \( n \in \mathbb{N} \),
\begin{align*}
z P^{\mathsf R}_n (z) &= P^{\mathsf R}_{n+1} (z) + P^{\mathsf R}_{n} (z) \xi^{\mathsf R}_n + P^{\mathsf R}_{n-1} (z) \eta^{\mathsf R}_n,
%&& n \in\N, \\
 &&
z Q^{\mathsf R}_n (z) = Q^{\mathsf R}_{n+1} (z) + Q^{\mathsf R}_{n} (z) \xi^{\mathsf R}_n + Q^{\mathsf R}_{n-1} (z) \eta^{\mathsf R}_n,
%&& n \in\N,
\end{align*}
where \( \eta^{\mathsf R}_n
%\coloneqq
 = C_n \eta^{\mathsf L}_n C_n^{-1} = C_{n-1} C_n^{-1}\), 
with initial conditions 
%\begin{align*}
\(
{P}^{\mathsf R}_{-1} = \0\), 
%&& 
\(
{P}^{\mathsf R}_{0} = \operatorname I \),
% && 
\(
Q_{-1}^{\mathsf R}(z) = -C_{-1}^{-1}\)
% && \text{
and
%} && 
\(Q_{0}^{\mathsf R}(z) = S_{W}(z) \),
%\end{align*}
and the relation \( \xi_n^{\mathsf R} 
 %\coloneqq
 = C_n \xi^{\mathsf L}_n C_n^{-1}\)
holds.

These objects are collected in the matrices,
\begin{align}
Y^{\mathsf L}_n (z) & 
 \coloneqq 
\begin{bmatrix}
 {P}^{\mathsf L}_{n} (z) & Q^{\mathsf L}_{n} (z) \\[.05cm]
-C_{n-1} {P}^{\mathsf L}_{n-1} (z) & -C_{n-1} Q^{\mathsf L}_{n-1} (z)
 \end{bmatrix}
 , \
\label{eq:ynl}
{Y}^{\mathsf R}_n (z) 
 \coloneqq 
\begin{bmatrix}
P^{\mathsf R}_{n} (z) & - P^{\mathsf R}_{n-1} (z) C_{n-1} \\[.05cm]
{Q}^{\mathsf R}_{n} (z) & - {Q}^{\mathsf R}_{n-1} (z) C_{n-1}
 \end{bmatrix} .
\end{align}
The three term recurrence relations for \( P^{\mathsf L}_{n}\), \(P^{\mathsf R}_{n}\), \( Q^{\mathsf L}_{n}\) and~\( Q^{\mathsf R}_{n}\) can be expressed as
%~follows,
\begin{align*}
Y_{n+1}^{\mathsf L} (z) & 
 = T^{\mathsf L}_n (z) Y^{\mathsf L}_n (z),&&
 Y_{n+1}^{\mathsf R} (z)
= Y^{\mathsf R}_n (z) T^{\mathsf R}_n (z) ,
 &&
n \in \mathbb N ,
\end{align*}
where the transfer matrices \(T^{\mathsf L}_n\) and \(T^{\mathsf R}_n\) are given by,
\begin{align*}
 T^{\mathsf L}_n=\begin{bmatrix}
z \operatorname I - \xi_n^{\mathsf L} & C_{n}^{-1} \\
-C_{n} & \0
\end{bmatrix},
&& T^{\mathsf R}_n
=
\begin{bmatrix}
z \operatorname I - \xi_n^{\mathsf R} & -C_{n} \\
C_{n}^{-1} & \0
\end{bmatrix}.
\end{align*}
Moreover, it is known, cf.~\cite{zbMATH07310648}, that
\begin{align*}
\big( {Y}^{\mathsf L}_n (z)\big)^{-1} = 
 \begin{bmatrix}
\0&\operatorname I\\
-\operatorname I &\0
 \end{bmatrix}
 {Y}^{\mathsf R}_n(z)
 \begin{bmatrix}
\0&-\operatorname I\\
\operatorname I &\0
 \end{bmatrix} .
\end{align*}

The structure of this work is as follows: In \S~\ref{sec:2} we formulate the Riemann--Hilbert problem for the matrix Bessel weights. Next, we study the analytic properties of the constant jump fundamental matrix associated with a matrix weight function that satisfies a Pearson-type equation of Bessel type. In \S~\ref{sec:3} we employ a Riemann--Hilbert approach to derive first and second-order differential equations from the structure matrix. We show how these equations reduce to well-known second-order differential equations in the scalar case. We also studied a nontrivial example of classical Bessel type matrix weight. Additionally, in \S~\ref{sec:4}, we~identify a discrete nonlinear relation for the recursion coefficients, which extends the concept of discrete Painlevé equations, by deriving an explicit expression for the structure matrix, and we illustrate our findings with an explicit example of Bessel type matrix weight.

\section{Riemann--Hilbert Formulation for MBPs}\label{sec:2}

Now, we return to the matrix Bessel weights
%defined
in Definition~\ref{def:1}.
Since the measures \( W^{(j,k)} \), \( j,k \in \{ 1, \ldots , N \} \) are H\"{o}lder continuous,
applying the Plemelj's formula, cf.~\cite{zbMATH03227290}, yields the following fundamental jump conditions,
\begin{align} \label{eq:inverseformula}
\big( Q^{\mathsf L}_n (z) \big)_+ - \big( Q_n^{\mathsf L} (z) \big)_- &= {P}^{\mathsf L}_n (z) W (z), 
 &&
\big( Q^{\mathsf R}_n (z) \big)_+ - \big( Q^{\mathsf R}_n (z) \big)_- = W (z){P}^{\mathsf R}_n (z) ,
\end{align}
for \( z\in\gamma\), where, \( \big( f(z) \big)_{\pm} = \lim_{\epsilon \to 0^{\pm}} f(z + \ii \hspace{-.045cm} \epsilon ) \), with \( \pm\) indicating the regions according to the ori\-en\-ta\-tion of the curve \( \gamma \).

Now, we state a theorem on Riemann--Hilbert problem for the Bessel type weights.
\begin{teo} \label{teo:LRHP_Bessel}
Given a regular Bessel type matrix of weights \( W \) with support on 
%\( (0,+\infty)\)
\( \gamma \), given in Definition~\ref{def:1}, we have the matrix function \( Y^{\mathsf L}_n\) and \( Y^{\mathsf R}_n\), defined 
in~\eqref{eq:ynl} are, for each \( n \in \N\), the unique solution of the following Riemann--Hilbert problem, which consists, respectively, in the determination of a~\(2N \times 2 N\) complex matrix function~such~that,
\begin{enumerate}[\rm (RH1)]
\item
\( Y_n^{\mathsf L}\) and \( Y_n^{\mathsf R}\) are holomorphic in \( \C \setminus \gamma \),

\item
satisfies the jump condition
\begin{align*} 
\big( Y^{\mathsf L}_n (z) \big)_+
=
\big( Y^{\mathsf L}_n (z) \big)_- \,
 \begin{bmatrix} \operatorname I & W (z) \\ \0 & \operatorname I \end{bmatrix} , 
 &&
\big( Y^{\mathsf R}_n (z) \big)_+ 
=
 \begin{bmatrix}
 \operatorname I & \0 \\ 
 W (z) & \operatorname I 
 \end{bmatrix} \big( Y^{\mathsf R}_n (z) \big)_- ,
 && z \in \gamma ,
\end{align*}
\item
have the following asymptotic behavior,
as \( |z| \to \infty\)
\begin{align*} 
 Y_n^{\mathsf L} (z)
 =
\Big(1+\operatorname{O}\big(\frac{1} {z}\big)\Big)
 \begin{bmatrix}z^{n}\operatorname I & \0 \\ 
 \0 & z^{-n} \operatorname I 
 \end{bmatrix} , 
&&
Y_n^{\mathsf R} (z) = 
 \begin{bmatrix}
 z^n \operatorname I & \0 \\ 
\0 & z^{-n} \operatorname I 
 \end{bmatrix}
\Big( 1 + \operatorname{O}\big(\frac{1}{z}\big) \Big) ,
\end{align*}
\item
\(
Y^{\mathsf L}_n (z) 
= \begin{bmatrix}
\operatorname{O} (1) & s^{\mathsf L}_{1}(z) \\[.05cm]
\operatorname{O} (1) & s^{\mathsf L}_{2}(z) 
 \end{bmatrix} \),
 \quad
 \(
Y^{\mathsf R}_n (z) 
= 
\begin{bmatrix}
\operatorname{O} (1) & \operatorname{O} (1) \\ 
s^{\mathsf R}_1(z) & s^{\mathsf R}_2(z) 
 \end{bmatrix} \), 
\quad as 
\quad
\(z \to 0\),
\quad
with
%\begin{align*}
\(
\lim_{z \to 0} z s^{\mathsf L}_j(z) = \0
 \), 
%&& 
\(
\lim_{z \to 0} z s^{\mathsf R}_j(z) = \0 
\),
%&&
\(
j=1,2
\).
%\end{align*}
\end{enumerate}
\end{teo}

\begin{proof}
{\rm (RH1)}, {\rm (RH2)}, and {\rm (RH3)} 
are direct consequences of the representation of the second kind functions~\eqref{eq:secondkindlaurent}
and the inverse formulas~\eqref{eq:inverseformula}.

We will prove~{\rm (RH4)} and the unicity of the solution.
The entries \( W^{j,k}\) of the matrix measure \( W \) are given in~\eqref{eq:the_weights_Bessel}.
It holds (cf.~\cite{zbMATH03227290}) that in a neighborhood of \( z=0\) the Cauchy transform
\begin{align*}
\phi_{m}(z) = 
\int_\gamma \frac{ p(t) A_m(t) t^{\alpha_{m}} \log^{p_{m}} (t) \Exp{-\frac{\lambda_m}{t}}}{t-z}
\frac{ \d t}{2 \pi \ii },
\end{align*}
where \(p(t)\) denotes any polynomial in \(t \),
that satisfies
 \( 
%\displaystyle
\lim_{z \to 0} z \phi_{m}(z) =0
\).
Then,~\({\rm (RH4)}\) is fulfilled by the matrices \( Y^{\mathsf L}_n,Y^{\mathsf R}_n\), respectively.
To prove the unicity of both Riemann--Hilbert problems let~us consider the matrix~function
\begin{align*}
G(z) = {Y}^{\mathsf L}_n (z) 
\begin{bmatrix}
\0 & \operatorname I \\
- \operatorname I &\0 
\end{bmatrix}
{Y}^{\mathsf R}_n(z)
\begin{bmatrix}
\0 &- \operatorname I \\
\operatorname I &\0 
\end{bmatrix}
 .
\end{align*}
The function \( G(z) \) has no jump or discontinuity on the curve \( \gamma \) and its behavior at
\( 0 \) is given by
\begin{align*}
G(z) &
 \sim{}
\begin{bmatrix}
 s^{\mathsf L}_1(z) + s^{\mathsf R}_2(z) & s^{\mathsf L}_1(z) + s^{\mathsf R}_1(z) \\[.05cm]
 s^{\mathsf L}_2(z) + s^{\mathsf R}_2(z) & s^{\mathsf L}_2(z) + s^{\mathsf R}_1(z)
\end{bmatrix}, & 
z&\to 0,
\end{align*}
so it holds that
\( 
%\displaystyle
\lim_{z \to 0} z G(z) = 0
\)
and we conclude that 
\( 0 \) is a removable singularity of \( G\). 
Now, from the behavior at infinity,
\begin{align*}
G(z) \sim{} 
\begin{bmatrix}
\operatorname I z^n & \0 \\
\0 & \operatorname I z^{-n} 
\end{bmatrix}
\begin{bmatrix}
\0 &\operatorname I \\
-\operatorname I &\0 
\end{bmatrix}
\begin{bmatrix}
 \operatorname I z^n & \0 \\
 \0 & \operatorname I z^{-n} 
 \end{bmatrix}
\begin{bmatrix}
\0 &-\operatorname I \\
\operatorname I &\0 
\end{bmatrix} =
\begin{bmatrix}
\operatorname I &\0 \\
\0 & \operatorname I
\end{bmatrix} ,
\end{align*}
hence the Liouville theorem implies that \( G(z)= \operatorname I \).
To prove the unicity of the solution, we consider another solution
\( {\widetilde{Y}^{\mathsf L}}_n\)
of the left Riemann--Hilbert problem. Then,
\begin{align*}
{\widetilde{Y}^{\mathsf L}}_n(z)=\left(
\begin{bmatrix}
\0 & \operatorname I \\
-\operatorname I & \0
\end{bmatrix} Y^{\mathsf R}_n(z)
\begin{bmatrix}
\0 & -\operatorname I \\
\operatorname I & \0
\end{bmatrix}\right)^{-1} .
\end{align*}
Hence any solution of this left Riemann--Hilbert problem is equal to the inverse of a fixed matrix, and the uniqueness follows. We obtain the uniqueness of the solution of the right Riemann--Hilbert in a similar way.
\end{proof}

Here we consider matrix of weights, \( W \), satisfying a matrix Pearson type equation
\begin{align} \label{eq:Bessel_Pearson}
z^2 W^\prime (z) = h^{\mathsf L} (z) W (z) + W (z) h^{\mathsf R} (z) ,
\end{align}
with entire matrix functions \( h^{\mathsf L} \), \( h^{\mathsf R} \).
If we take a matrix function \( W_{\mathsf L}\) such that
\begin{align} \label{eq:partial_Bessel_Pearson1}
z^2 W_{\mathsf L}^\prime (z) = h^{\mathsf L} (z) W_{\mathsf L} (z) ,
\end{align}
then there exists a matrix function \( W_\mathsf R (z)\) such that \( W (z) = W_{\mathsf L}(z) W_{\mathsf R} (z)\) with
\begin{align} \label{eq:partial_Bessel_Pearson2}
z^2 W_{\mathsf R}^\prime (z) = W_{\mathsf R} (z) h^{\mathsf R} (z) .
\end{align}
The reciprocal is also true.

The solution of~\eqref{eq:partial_Bessel_Pearson1} and~\eqref{eq:partial_Bessel_Pearson2} will have possibly branch points
at \( 0 \), cf.~\cite{zbMATH07114556}.
This~means that exists constant matrices, \( \mathsf C^\mathsf L\), \( \mathsf C^\mathsf R\) such that
\begin{align}
\label{eq:salto0}
(W_{\mathsf L} (z))_- = (W_{\mathsf L} (z))_+ \mathsf C^\mathsf L, &&
(W_{\mathsf R} (z))_- = \mathsf C^\mathsf R (W_{\mathsf R} (z))_+, && \text{in } & \
\gamma .
\end{align}
We introduce, the \emph{constant jump fundamental matrices}, for all \( n \in \mathbb N\),
\begin{align}
\label{eq:zn1}
Z_n^{\mathsf L}(z) 
& \coloneqq Y^{\mathsf L}_n (z) 
\begin{bmatrix} 
W_{\mathsf L} (z) & \0 \\ 
\0 & W_{\mathsf R}^{-1} (z)
 \end{bmatrix} ,
 &&
{Z}^{\mathsf R}_n (z) \coloneqq
\begin{bmatrix} 
W_{\mathsf R} (z) & \0 \\ 
\0 & W_{\mathsf L}^{-1} (z) 
 \end{bmatrix}
{Y}^{\mathsf R}_n (z) .
\end{align}

\begin{teo}
The constant jump fundamental matrices~\( Z^{\mathsf L}_n\) and \( Z^\mathsf R_n\) satisfy, for each \( n \in \N\), the fol\-low\-ing~properties,
\begin{enumerate}[\rm i)]

\item
are holomorphic on \( \C \setminus \gamma \),
\item
present the fol\-low\-ing \emph{constant jump condition} on \( \gamma \)
\begin{align*}
\big( Z^{\mathsf L}_n (z) \big)_+ &
= \big( Z^{\mathsf L}_n (z) \big)_- 
\begin{bmatrix} 
\mathsf C^{\mathsf L} & \mathsf C^{\mathsf L}
 \\
\0 & \operatorname I 
\end{bmatrix}, 
 &&
\big( {Z}^{\mathsf R}_n (z) \big)_+ 
=
\begin{bmatrix} 
\operatorname I & \0
 \\
\mathsf C^{\mathsf R} & \mathsf C^{\mathsf R}
\end{bmatrix}
\big( {Z}^{\mathsf R}_n (z) \big)_- .
\end{align*}
\end{enumerate}
\end{teo}

\begin{proof}
From the definition of \( Z^{\mathsf L}_n (z)\) we have
\begin{align*}
\big(Z^{\mathsf L}_n (z) \big)_+ = \big( Y^{\mathsf L}_n (z) \big)_+ 
\begin{bmatrix}
(W_{\mathsf L} (z))_+ & \0 \\ 
\0 & ( W_{\mathsf R}^{-1} (z))_+
 \end{bmatrix} , 
\end{align*}
and taking into account Theorem~\ref{teo:LRHP_Bessel} we obtain
\begin{align*}
\big(Z^{\mathsf L}_n (z) \big)_+ &
= \big( Y^{\mathsf L}_n (z) \big)_-
\begin{bmatrix} 
\operatorname I & (W_{\mathsf L} (z) W_{\mathsf R} (z))_+ \\ 
\0 & \operatorname I
 \end{bmatrix}
\begin{bmatrix} 
(W_{\mathsf L} (z))_+ & \0 \\ 
\0 & ( W_{\mathsf R}^{-1} (z))_+ 
 \end{bmatrix} \\
&
 =
\resizebox{.815\hsize}{!}{\(
\big( Y^{\mathsf L}_n (z) \big)_-
\begin{bmatrix} 
(W_{\mathsf L} (z))_- & \0 \\ 
\0 & ( W_{\mathsf R}^{-1} (z))_- 
 \end{bmatrix}
\begin{bmatrix} 
(W_{\mathsf L}^{-1} (z))_- & \0 \\ 
\0 & ( W_{\mathsf R} (z))_- 
 \end{bmatrix}
\begin{bmatrix} 
(W_{\mathsf L} (z))_+ & (W_{\mathsf L} (z))_+ \\ 
\0 & ( W_{\mathsf R}^{-1} (z))_+ 
 \end{bmatrix} \)}
 \\
&
= 
\big(Z^{\mathsf L}_n (z) \big)_- 
\begin{bmatrix} 
(W_{\mathsf L}^{-1} (z))_- (W_{\mathsf L} (z))_+ &
(W_{\mathsf L}^{-1} (z))_- (W_{\mathsf L} (z))_+ \\ 
\0 & 
W_{\mathsf R} (z)_- ( W_{\mathsf R}^{-1} (z))_+ 
 \end{bmatrix} ,
\end{align*}
as we wanted to prove.
\end{proof}

In parallel to the matrices \( Z^{\mathsf L}_n(z)\) and \( Z^{\mathsf R}_n(z)\), for each factorization we introduce what we call \emph{structure matrices} given in terms of the \emph{left}, respectively \emph{right}, logarithmic derivatives by,
\begin{align} \label{eq:Mn}
M^{\mathsf L}_n (z) & 
 \coloneqq \big(Z^{\mathsf L}_n\big)^{\prime} (z) \big(Z^{\mathsf L}_{n} (z)\big)^{-1},&
{M}^{\mathsf R}_n (z) & 
 \coloneqq \big ({Z}^{\mathsf R}_{n} (z)\big)^{-1} \big(Z^{\mathsf R}_n\big)^{\prime} (z) .
\end{align}
It is not difficult to see that
\begin{align} \label{MnLR}
{M}^{\mathsf R}_{n}(z) &
 =
\begin{bmatrix}
0 & \operatorname I \\ -\operatorname I & 0
 \end{bmatrix} 
M_n^{\mathsf L} (z)
\begin{bmatrix}
0 & \operatorname I \\ -\operatorname I & 0
 \end{bmatrix} , &
n \in \mathbb N ,
\end{align}
as well as, the following properties hold~(cf.~\cite{zbMATH07310648}),
\begin{enumerate}[\rm i)]

\item
the transfer matrices satisfy
%\begin{align*}
\(
 T^{\mathsf L}_n (z)Z_n^{\mathsf L}(z) = Z^{\mathsf L}_{n+1}(z) 
 \),
\(
 {Z}^{\mathsf R}_n(z){T}^{\mathsf R}_n (z) =
 {Z}^{\mathsf R}_{n+1}(z)\), 
\( n \in \mathbb N \),
% \end{align*}

\item 
the zero curvature formulas holds for all \( n \in \mathbb N\),
\begin{align*}
\big( T^{\mathsf L}_n \big)^\prime (z) & 
= M^{\mathsf L}_{n+1} (z) T^{\mathsf L}_n(z) - T^{\mathsf L}_n (z) M^{\mathsf L}_{n} (z),
 &&
\big( T^{\mathsf R}_n \big)^\prime (z) =
T^{\mathsf R}_n (z) \, M^{\mathsf R}_{n+1} (z) 
- M^{\mathsf R}_{n} (z) T^{\mathsf R}_n (z) .
\end{align*}

\end{enumerate}
Now, we discuss the holomorphic properties of the structure matrices just introduced.

\begin{teo} \label{teo:mn_Bessel}
Let \( W \) be a regular Bessel matrix weight that admits a factorization~\( W (z) = W_{\mathsf L} (z) W_{\mathsf R} (z)\), where \( W_{\mathsf L}\) and \( W_{\mathsf R}\) satisfies~\eqref{eq:partial_Bessel_Pearson1} and~\eqref{eq:partial_Bessel_Pearson2}.
Then, the structure matrices~\( M^{\mathsf L}_n(z) \) and \( M^\mathsf R_n(z)\) are meromorphic on~\( \C \) for each \( n \in \N\), with a singularity at \( z=0\), which is at most a pole of order 2.
\end{teo}

\begin{proof}
Let us prove the statement for \( M^{\mathsf L}_n(z) \).
The matrix function \( M^{\mathsf L}_n(z) \) is holomorphic in~\( \C \setminus
\gamma
\) by definition,~cf.~\eqref{eq:Mn}.
Due to the fact that \( Z^{\mathsf L}_n(z)\) has a constant jump on
 \( \gamma \),~cf.~\eqref{eq:salto0}, the matrix function
 \( 
%\displaystyle 
\big(Z^{\mathsf L}_n\big)^{\prime} 
\) has the same constant jump on \( \gamma \), so~that the matrix \( M^{\mathsf L}_n \) has no jump on~\( \gamma \), and it follows that
at \( z=0\), \( M^{\mathsf L}_n \) has an isolated singularity. 

From~\eqref{eq:zn1}~and~\eqref{eq:Mn} it holds
\begin{align}
\label{eq:numerar_Bessel}
M^{\mathsf L}_n (z)
& 
= \big(Y^{\mathsf L}_n\big)^{\prime} (z)\big( {Y^{\mathsf L}_n} (z) \big)^{-1} 
+ \frac{1}{z^2} Y^{\mathsf L}_n (z)
 \begin{bmatrix}
h^{\mathsf L}(z) & \0 \\
\0 & -h^{\mathsf R} (z)
 \end{bmatrix}
\big( {Y^{\mathsf L}_n} (z)\big)^{-1}, 
\end{align}
where
 \( Y^{\mathsf L}_n\) is given in~\eqref{eq:ynl}. 
Each entry of the matrix \( Q^{\mathsf L}_{n} (z)\) is the Cauchy transform of certain function,~\( f \), of type
\begin{align*}
f(z)=\sum_{j \in I} \varphi_{j}(z) z^{p_{j}} \log^{q_{j}} (z)\Exp{-\frac{\lambda_j}{z}} ,
\end{align*}
where \( \varphi_{j} (z)\) is, for each \( j \in I\), an entire function with \( \Re(\lambda_{j})\geq 0\), 
\( q_j \in \mathbb{N}\), and \( I \) is a finite set of indices.
It's clear that
\( 
%\displaystyle
\lim_{z \to 0} z f(z)= 0 
\).
By~\cite[\S 8.3-8.6]{zbMATH03227290} and~\cite{zbMATH03277891}, we deduce that the Cauchy transform of \( f \), i.e.
\( 
%\displaystyle
g(z) = \int_{\gamma} \frac{f(t)}{t - z} \frac{\d t} {2\pi\ii}
\),
has the same~property,
\( 
%\displaystyle
\lim_{z \to 0} 
z g(z)
 = 0 \).
We can also see that 
 \( 
%\displaystyle
\lim_{z \to 0} z^3 g^{\prime} (z)= 0
\).
Indeed,
\begin{align*}
z^2 g^{\prime} (z)&= 
\int_\gamma \frac{z^2 f(t)}{(t-z)^2} \, \frac{\d t} {2\pi\ii} = \int_\gamma \frac{(z-t)(z+t) f(t)}{(t-z)^2} \, \frac{\d t} {2\pi\ii} + \int_\gamma \frac{t^2 f(t)}{(t-z)^2} \, \frac{\d t} {2\pi\ii} ,\\
&= -\int_\gamma \frac{(z+t)f(t)}{t-z} \, \frac{\d t} {2\pi\ii} - \frac{1} {2\pi\ii}\frac{t^2 f(t)}{t - z} \bigg|_{\partial \gamma} 
+ \int_\gamma \frac{(t^2f(t))^{\prime}}{t-z} \, \frac{\d t} {2\pi\ii} \\
&= -z \int_\gamma \frac{f(t)}{t-z} \, \frac{\d t} {2\pi\ii} 
+ \int_\gamma \frac{t^2 f^{\prime}(t)}{t-z} \, \frac{\d t} {2\pi\ii} + \int_\gamma \frac{t f(t)}{t-z} \, \frac{\d t} {2\pi\ii} ,
\end{align*}
taking into account the boundary conditions,
\( \frac{t^2 f(t)}{t - z} \Big|_{\partial \gamma} = 0 \).
From this, and using the definition of \( f \) we get that \( t^2 f^{\prime}(t)\) and \( t f(t)\) are functions in the class of \( f \).
Furthermore,~the Cauchy transform of \( f \) satisfy the same property
\( 
%\displaystyle
\lim_{z\to 0} z g(z)=0\), consequently,
 \(
%\displaystyle
\lim_{z \to 0} z^3 g^{\prime} (z) 
= 0\).
From~these considerations it follows,
\begin{align*}
\big(Y^{\mathsf L}_n\big)^{\prime} (z)
 &
= \begin{bmatrix}
\operatorname{O} (1) & r^{\mathsf L}_1 (z) \\[.05cm]
\operatorname{O} (1) & r^{\mathsf L}_2 (z)
\end{bmatrix} ,
 & 
\big( {Y^{\mathsf L}_n(z)}\big)^{-1} 
 &
= \begin{bmatrix}
 r^{\mathsf L}_3(z) & r^{\mathsf L}_4 (z) \\[.05cm]
\operatorname{O} (1) &\operatorname{O} (1)
\end{bmatrix} , & 
z &
\to 0,
\end{align*}
where 
 \(
%\displaystyle 
\lim_{z \to 0} z^2 r^{\mathsf L}_i(z) = \0 \), 
for \( i = 1,2\), and 
 \(
%\displaystyle
\lim_{z \to 0} z^2 r^{\mathsf R}_i(z) = \0 \),
for \( i = 3,4\), so it holds~that
\begin{align*}
\lim_{z \to 0} 
z^3 \big(Y^{\mathsf L}_n\big)^{\prime} (z) \big({Y^{\mathsf L}_n}\big)^{-1} 
= \lim_{z \to 0} z^3
\begin{bmatrix}
 r^{\mathsf L}_1(z) + r^{\mathsf L}_3(z)& r^{\mathsf L}_1(z) + r^{\mathsf L}_4(z) \\[0.05cm]
 r^{\mathsf L}_2(z) + r^{\mathsf L}_3(z)& r^{\mathsf L}_2(z) + r^{\mathsf L}_4(z)
\end{bmatrix}
=\0.
\end{align*}
Similar considerations leads us to
\begin{align*}
\lim_{z \to 0} z \, Y^{\mathsf L}_n(z)
\begin{bmatrix}
h^{\mathsf L}(z) & \0 \\
\0 & -h^{\mathsf R}(z)
\end{bmatrix}
\big( {Y^{\mathsf L}_n(z)}\big)^{-1} =\0 ,
\end{align*}
so we obtain that
 \(
%\displaystyle
\lim_{z\to 0} z^3 M^{\mathsf L}_n(z) =\0 \),
and hence the matrix function \( M^{\mathsf L}_n(z) \) has at most a double pole on~\( z=0\).
By analogous arguments we get the results for \( M_n^{\mathsf R}\).
\end{proof}

\section{Differential Equations: Insights and Applications}\label{sec:3}

Our objective is to derive differential equations satisfied by the biorthogonal matrix polynomials associated to regular Bessel type matrices of weights.

Let us define a new matrix functions,
%\begin{align*}
\(
\widetilde{M}^\mathsf L_n (z) = z^2{M}^\mathsf L_n (z) 
\),
\(
\widetilde{M}^\mathsf R_n (z) = z^2{M}^\mathsf R_n (z) 
\),
%\end{align*}
then \( \widetilde{M}^\mathsf L_n (z)\) and \( \widetilde{M}^\mathsf R_n (z)\) are matrices of entire functions, cf. Theorem~\ref{teo:mn_Bessel}. 

\begin{pro} \label{prop:mn_Bessel}
Let \( W \) be a regular Bessel matrix weight
that admits a 
%with
 factorization~\( W (z) = W_{\mathsf L} (z) W_{\mathsf R} (z)\), where \( W_{\mathsf L}\) and \( W_{\mathsf R}\) satisfies~\eqref{eq:partial_Bessel_Pearson1} and~\eqref{eq:partial_Bessel_Pearson2}.
Then, for all \( n \in \mathbb N \), we have,
\begin{align}
\label{eq:firstodeYnL_Bessel}
z^2 \big(Y^{\mathsf L}_n\big)^{\prime} (z) + Y^{\mathsf L}_n (z)
\begin{bmatrix}
h^{\mathsf L} (z) & \0 \\ \0 & - h^{\mathsf R} (z)
 \end{bmatrix}
&
= \widetilde{M}^{\mathsf L}_n (z) Y^{\mathsf L}_n (z)
 \\
\label{eq:firstodeYnR_Bessel}
z^2 \big(Y^{\mathsf R}_n\big)^{\prime} (z)
 +
\begin{bmatrix}
h^{\mathsf R} (z) & \0 \\ \0 & - h^{\mathsf L} (z)
 \end{bmatrix}
Y^{\mathsf R}_n (z)
&
= Y^{\mathsf R}_n (z) \widetilde{M}^{\mathsf R}_n (z) .
\end{align}
\end{pro}

\begin{proof}
Equations~\eqref{eq:firstodeYnL_Bessel} and~\eqref{eq:firstodeYnR_Bessel} follows immediately from the definition of the matrices~\( M^{\mathsf L}_n(z)\) and \( M^{\mathsf R}_n(z)\) in~\eqref{eq:Mn}. 
\end{proof}

Now, we introduce the Miura-like map \( \mathcal B \), 
 \(
%\displaystyle 
\mathcal B (F(z)) = F^{\prime}(z)+\frac{F^2(z)}{z^2}
\).

\begin{pro} \label{prop:mnn_Bessel}
Let \( W \) be a regular Bessel matrix weight 
that admits a 
%with
factorization~\( W (z) = W_{\mathsf L} (z) W_{\mathsf R} (z)\), where \( W_{\mathsf L}\) and \( W_{\mathsf R}\) satisfies~\eqref{eq:partial_Bessel_Pearson1} and~\eqref{eq:partial_Bessel_Pearson2}.
Then,
\begin{align} \label{eq:edo1_Bessel}
\resizebox{.9095\hsize}{!}{\( 
z^2 \big(Y^\mathsf L_n\big)^{\prime\prime} 
+ \big(Y^\mathsf L_n\big)^{\prime} 
\begin{bmatrix} 
 2 h^{\mathsf L}+ 2z \operatorname I& \0 \\ 
 \0 & - 2 h^{\mathsf R} +2z \operatorname I
 \end{bmatrix}
 + Y_n^\mathsf L (z)
 \begin{bmatrix}
 \mathcal B (h^{\mathsf L}) & \0 \\ 
 \0 &\mathcal B (- h^{\mathsf R}) 
 \end{bmatrix}
= \mathcal B ( \widetilde{M}^\mathsf L_n )Y^\mathsf L_n 
\)} ,
 \\
\label{eq:edo2_Bessel}
\resizebox{.9095\hsize}{!}{\(
z^2 \big(Y^\mathsf R_n\big)^{\prime\prime} 
+
\begin{bmatrix} 
2 h^{\mathsf R} + 2z \operatorname I & \0 \\ 
 \0 & - 2 h^{\mathsf L} + 2z \operatorname I 
 \end{bmatrix} 
 \big(Y^\mathsf R_n\big)^{\prime} 
 + 
\begin{bmatrix} 
 \mathcal B (h^{\mathsf R}) & \0 \\ 
 \0 & \mathcal B (-h^{\mathsf L}) 
 \end{bmatrix}
Y_n{^\mathsf R} (z)
 = 
Y^\mathsf R_n \mathcal B (\widetilde{M}^\mathsf R_n )
\)}.
 \end{align}
\end{pro}

\begin{proof}
Differentiating in~\eqref{eq:Mn} we get
%\begin{align*}
\(
\big(Z_{n}^{\mathsf L}\big)^{\prime \prime}\big(Z_{n}^{\mathsf L}\big)^{-1}=\frac{\big(\widetilde{M}_{n}^{\mathsf L}\big)^{\prime}}{z^2}-2\frac{\widetilde{M}_{n}^{\mathsf L}}{z^{3}}+\frac{\big(\widetilde{M}_{n}^{\mathsf L}\big)^{2}}{z^{4}}
\),
 %&& \text{
 or equivalently
 %} &&
\( 
z^2\big(Z_{n}^{\mathsf L}\big)^{\prime \prime}\big(Z_{n}^{\mathsf L}\big)^{-1}+ 2zM_n^{\mathsf L}
= \mathcal B (\widetilde{M}^\mathsf L_n ) \).
%\end{align*}
From~\eqref{eq:firstodeYnL_Bessel}, we get
\begin{align*}
2zM_n^{\mathsf L}=2z\big(Y_n^{\mathsf L}\big)^\prime{(Y_n^{\mathsf L})}^{-1} + \frac{2}{z} Y_n^{\mathsf L}
\begin{bmatrix}
h^{\mathsf L}&\0
 \\
\0 & -h^{\mathsf R}
\end{bmatrix}
{(Y_n^{\mathsf L})}^{-1} .
\end{align*}
Furthermore, \eqref{eq:partial_Bessel_Pearson1} and \eqref{eq:partial_Bessel_Pearson2} leads, respectively, to
%\begin{align*}
\(
z^2{\big(W_{\mathsf L}\big)}^{\prime\prime}{\big(W_{\mathsf L}\big)}^{-1}=\frac{{(h^{\mathsf L})}^2}{z^2}-\frac{2}{z}h^{\mathsf L}+{(h^{\mathsf L})}^\prime
\)
% \quad \text{ 
and
% } \quad
\(
z^2{\big((W_{\mathsf R})^{-1}\big)}^{\prime\prime}W_{\mathsf R}=\frac{({h^{\mathsf R})}^2}{z^2}+\frac{2}{z}h^{\mathsf R}-{(h^{\mathsf R})}^\prime 
\).
%\end{align*}
Differentiating again~\eqref{eq:Mn} and replacing~\( Z_{n}^{\mathsf L}\) from~\eqref{eq:zn1}
\begin{multline*}
z^2\big(Z_{n}^{\mathsf L}\big)^{\prime \prime}{\big(Z_{n}^{\mathsf L}\big)}^{-1}=z^2{(Y_n^{\mathsf L})}^{\prime\prime}Y_n^{\mathsf L}+{\big(Y_n^{\mathsf L}\big)}^\prime
\begin{bmatrix}
2h^{\mathsf L}& \0 \\
\0 & -2h^{\mathsf R}
\end{bmatrix}
{(Y_n^{\mathsf L})}^{-1} \\
+Y_n^{\mathsf L}
\begin{bmatrix}
{z^2\big(W_{\mathsf L}\big)}^{\prime\prime}{\big(W_{\mathsf L}\big)}^{-1} & \0 \\
\0 & {z^2\big((W_{\mathsf R})^{-1}\big)}^{\prime\prime}W_{\mathsf R} 
\end{bmatrix}
{(Y_n^{\mathsf L})}^{-1}.
\end{multline*}
Finally, by considering equations above, we get the stated result~\eqref{eq:edo1_Bessel}. The equation~\eqref{eq:edo2_Bessel} follows in a similar way from definition of~\( M_n^{\mathsf R}\) in~\eqref{eq:Mn}.
\end{proof}

First of all we say that a weight matrix is semiclassical when the weight matrix satisfy a generalized Pearson equation of type
\eqref{eq:Bessel_Pearson}, with \( h^\mathsf L \) and \( h^\mathsf R \) are functions of polynomial type.
The~case when \( h^\mathsf L \) and \( h^\mathsf R \) are polynomials of degree one is of classical type.

\begin{pro}
Let \( W \) be a regular Bessel matrix weight
that admits a 
%with 
factorization
\( W (z) = W_{\mathsf L} (z) W_{\mathsf R} (z)\), where \( W_{\mathsf L}\) and \( W_{\mathsf R}\) satisfies~\eqref{eq:partial_Bessel_Pearson1} and~\eqref{eq:partial_Bessel_Pearson2}.
If \( h^{\mathsf L} (z)=A^{\mathsf L} z+B^{\mathsf L}\)
and \( h^{\mathsf R}(z)=A^{\mathsf R} z+B^{\mathsf R} \), 
then the left and right fundamental matrices are, for each \(n\in \mathbb N \), given by,
\begin{align}
\label{generalMnL_Bessel}
 \widetilde{M}^{\mathsf L}_n (z) & = 
\resizebox{.805\hsize}{!}{\(
\begin{bmatrix} 
\big(A^{\mathsf L}+n \operatorname I \big) z+\big[p^1_{\mathsf L,n},A^\mathsf L\big] - p^1_{\mathsf L,n} + B^{\mathsf L}
 &
A^{\mathsf L} C_n^{-1} + C_n^{-1} A^{\mathsf R}+(2n+1)C_{n}^{-1}
 \\
-C_{n-1} A^{\mathsf L} -A^{\mathsf R} C_{n-1}-(2n-1)C_{n-1}
&
-\big(n \operatorname I +A^{\mathsf R}\big) z+\big[p^1_{\mathsf R,n},A^\mathsf R\big] + p^1_{\mathsf R,n} - B^{\mathsf R}
 \end{bmatrix}
 \)},
 \\
\label{generalMnR_Bessel}
\widetilde{M}^{\mathsf R}_n (z) & =
\resizebox{.805\hsize}{!}{\(
\begin{bmatrix}
\big(n \operatorname I +A^{\mathsf R}\big) z-\big[p^1_{\mathsf R,n},A^\mathsf R\big] - p^1_{\mathsf R,n} + B^{\mathsf R} &-C_{n-1} A^{\mathsf L} -A^{\mathsf R} C_{n-1}-(2n-1)C_{n-1} \\
A^{\mathsf L} C_n^{-1} + C_n^{-1} A^{\mathsf R}+(2n+1)C_{n}^{-1}
 & 
-\big(A^{\mathsf L}+n \operatorname I \big) z -\big[p^1_{\mathsf L,n},A^\mathsf L\big] + p^1_{\mathsf L,n} - B^{\mathsf L}
\end{bmatrix}
\)} .
\end{align} 
\end{pro}

\begin{proof}
Taking \( \vert z\vert \to +\infty \) in~\eqref{eq:numerar_Bessel} we have that
\begin{align*}
Y_n^\mathsf{L}
=
\resizebox{.805\hsize}{!}{\(
\begin{bmatrix}
\operatorname I z^{n}+p_{\mathsf{L}, n}^{1} z^{n-1}+\cdots & -C_{n}^{-1}\big(\operatorname I z^{-n-1}+q_{\mathsf{L}, n}^{1} z^{-n-2}+\cdots\big) \\
-C_{n-1}\big(\operatorname I z^{n-1}+p_{\mathsf{L}, n-1}^{1} z^{n-2}+\cdots\big) & \operatorname I z^{-n}+q_{\mathsf{L}, n-1}^{1} z^{-n-1}+\cdots
\end{bmatrix}
\)}
 ,
 \\
\big(Y_n^\mathsf{L}\big)^{-1}
=
\resizebox{.805\hsize}{!}{\(
\begin{bmatrix}
\operatorname I z^{-n}+q_{\mathsf{R}, n-1}^{1} z^{-n-1}+\cdots 
&\left(\operatorname I z^{-n-1}+q_{\mathsf{R}, n}^{1} z^{-n-2}+\cdots\right) C_{n}^{-1} \\
 \left(\operatorname I z^{n-1}+p_{\mathsf{R}, n-1}^{1} z^{n-2}+p_{\mathsf{R}, n-1}^{2} z^{n-3}+\cdots \right)C_{n-1} & \operatorname I z^{n}+p_{\mathsf{R}, n}^{1} z^{n-1}+p_{\mathsf{R}, n}^{2} z^{n-2}+\cdots
\end{bmatrix}
\)} 
 .
\end{align*}
Hence, as \( \vert z\vert \to +\infty \)
\begin{align*}
z^2\big(Y_{n}^{\mathsf L}\big)^{\prime}\big(Y_{n}^{\mathsf L}\big)^{-1}
=
\resizebox{.695\hsize}{!}{\(
\begin{bmatrix}
n \operatorname I z+ \big(n q_{\mathsf{R}, n-1}^{1}+(n-1) p_{\mathsf{L}, n}^{1} \big) & (2n+1)C_{n}^{-1}\\
-(2n-1)C_{n-1} &-n \operatorname I z - \big( n p_{\mathsf{R}, n}^{1}+(n+1)q_{\mathsf{L}, n-1}^{1}\big)
\end{bmatrix}
\)}
+ \operatorname O\big(\frac{1} {z}\big) .
\end{align*}
Moreover,
\begin{multline*}
Y_{n}^{\mathsf L}(z)
\begin{bmatrix}
h^{\mathsf L}(z) & \0 \\
\0 & -h^{\mathsf R}(z)
\end{bmatrix}
\big(Y_{n}^{\mathsf L}(z)\big)^{-1}
 \\ =
\begin{bmatrix}
A^{\mathsf L}z+A^{\mathsf L}q_{\mathsf R,n-1}^1+p_{\mathsf L,n}^1A^{\mathsf L}+B^{\mathsf L} & A^{\mathsf L}C_{n}^{-1}+C_{n}^{-1}A^{\mathsf R}\\
C_{n-1}A^{\mathsf L}+A^{\mathsf R}C_{n-1}& -A^{\mathsf R}z-A^{\mathsf R}p_{\mathsf R,n}^1-q_{\mathsf L,n-1}^1A^{\mathsf R}-B^{\mathsf R}
\end{bmatrix}
+ \operatorname O \big(\frac{1} {z} \big) .
\end{multline*}
Since
\( 
z^2\big(Y_{n}^{\mathsf L}\big)^{\prime}(z)\big(Y_{n}^{\mathsf L}(z)\big)^{-1}+Y_{n}^{\mathsf L}(z)
\begin{bmatrix}
h^{\mathsf L}(z) & \0 \\
\0 & -h^{\mathsf R}(z)
\end{bmatrix}
\big(Y_{n}^{\mathsf L}(z)\big)^{-1}\) is holomorphic on~\( \mathbb{C}\), by~Liouville theorem we deduce that,
\begin{multline*}
z^2\big(Y_{n}^{\mathsf L}\big)^{\prime}(z)\big(Y_{n}^{\mathsf L}(z)\big)^{-1} + Y_{n}^{\mathsf L}(z)
\begin{bmatrix}
h^{\mathsf L}(z) & \0 \\
\0 & -h^{\mathsf R}(z)
\end{bmatrix}
\big(Y_{n}^{\mathsf L}(z)\big)^{-1}
 \\ =
 \begin{bmatrix}
n \operatorname I z+ \big(n q_{\mathsf{R}, n-1}^{1}+(n-1) p_{\mathsf{L}, n}^{1}\big) & (2n+1)C_{n}^{-1}\\
-(2n-1)C_{n-1} & -n \operatorname I z - \big( np_{\mathsf{R}, n}^{1}+(n+1)q_{\mathsf{L}, n-1}^{1}\big)
\end{bmatrix} 
 \\ +
\begin{bmatrix}
A^{\mathsf L}z+A^{\mathsf L}q_{\mathsf R,n-1}^1+p_{\mathsf L,n}^1A^{\mathsf L}+B^{\mathsf L} & A^{\mathsf L}C_{n}^{-1}+C_{n}^{-1}A^{\mathsf R}\\
C_{n-1}A^{\mathsf L}+A^{\mathsf R}C_{n-1}& -A^{\mathsf R}z-A^{\mathsf R}p_{\mathsf R,n}^1-q_{\mathsf L,n-1}^1A^{\mathsf R}-B^{\mathsf R}
\end{bmatrix}
 .
\end{multline*}
By considering the identities \( p^1_{\mathsf R,n} = -q^1_{\mathsf L,n-1}\) and \(p^1_{\mathsf L,n} = -q^1_{\mathsf R,n-1}\), then~\eqref{generalMnL_Bessel} follows. The relation~\eqref{MnLR} leads to~\eqref{generalMnR_Bessel}.
\end{proof}

We aim to apply the previously performed computation~\eqref{generalMnL_Bessel} to derive familiar formulas in the context of scalar scenarios.

\begin{pro}
Let us consider the weight \( W(z)=z^{a-2} \Exp{-\frac{b}{z}}\), with
 \( a \in \mathbb R \setminus \{ -1,-2, \ldots \} \)
%is a positive integer 
 and \( b \in \mathbb C \) with positive real part.
% different from zero. 
 Then, the scalar second order equation for \( \big\{P^{\mathsf L}_n\big\}_{n\in\N}\)~(cf. for example~\cite{zbMATH03048838}) and~\( \big\{Q^{\mathsf L}_n\big\}_{n\in\N}\) is given by
\begin{align} \label{scalarsode_Bessel}
 & z^2 ( P_n^\mathsf L )^{\prime\prime} (z) + \big(az + b\big) ( P_n^\mathsf L )^{\prime} (z) - n(a+n-1) P_n^\mathsf L (z)=0 , \\ 
 \label{scalarsodeq_Bessel}
 & z^2 ( Q_n^\mathsf L )^{\prime\prime} (z) + \big((4-a)z -b\big) ( Q_n^\mathsf L )^{\prime} (z) -
%\big(n(a+n-1)+2-a \big) 
 ( n+1 ) ( n+a-2 )Q_n^\mathsf L (z)=0 .
\end{align}
\end{pro}

To the best of our knowledge, the equation~\eqref{scalarsodeq_Bessel} is novel.

\begin{proof} 
Tacking
%If we consider
\( W_{\mathsf L}=W_{\mathsf R}=z^{\frac{a}{2}-1} \Exp{-\frac{b}{2z}}\)
%then \( W(z)=W_{\mathsf L}W_{\mathsf R}=z^{a-2} \Exp{-\frac{b}{z}}\). 
in~\eqref{generalMnL_Bessel} we get
\begin{align*}
\widetilde{M}^{\mathsf L}_n (z)=
 \begin{bmatrix}
(\frac{a-2}{2}+n)z-p_n^1+\frac{b}{2}& C_n^{-1}(a+2n-1)\\
-C_{n-1}(a+2n-3)& -(\frac{a-2}{2}+n)z+p_n^1-\frac{b}{2}
 \end{bmatrix} .
\end{align*}
It is easy to see that 
%\begin{align*}
\(
\big( \widetilde{M}^{\mathsf L}_n (z)\big)^2=
\big((\frac{a-2}{2}+n)z-p_n^1+\frac{b}{2}\big)^2 -\gamma_n\big((a+2n-2)^2-1 \big)
\operatorname I
\),
%\end{align*}
and also,
\( \big( \widetilde{M}^{\mathsf L}_n (z)\big)^\prime=
\begin{bsmallmatrix}
\frac{a-2}{2}+n&0\\ 0 &-\frac{a-2}{2}+n
 \end{bsmallmatrix} \).
Using now Proposition~\ref{prop:mnn_Bessel} we get
\begin{multline*}
z^2 ( P_n^\mathsf L )^{\prime\prime} (z) + \big(az + b\big) ( P_n^\mathsf L )^{\prime} (z) - n(a+n-1) P_n^\mathsf L (z)
\\
 =\Big(\frac{nb-p_n^1(a+2n-2)}{z}
 +
 \frac{p_n^1\big(p_n^1-b \big)-\gamma_n\big((a+2n-2)^2-1 \big)}{z^2} \Big)
 P_n^\mathsf L (z) .
\end{multline*}
By equalizing poles between left and right hand side on \( 0 \) we obtain
\begin{align*}
\Big( p_n^1\big(p_n^1-b \big)-\gamma_n\big((a+2n-2)^2-1 \big)\Big)P_n^\mathsf L (0) =0
\end{align*}
which, taking into account \( P_n^\mathsf L (0)\neq 0\), gives
%\begin{align*}
\(
\big(nb-p_n^1(a+2n-2)\big)P_n^\mathsf L (0) =0 
\).
%\end{align*}
These two equations lead to the representation of \( p_n^1\) and~\( \gamma_n\), as well as~\eqref{scalarsode_Bessel}.
The~equation~\eqref{scalarsodeq_Bessel} for the \( \{ Q_n^\mathsf L\}_{n \in \mathbb N}\) also follows from the above considerations. 
\end{proof}

%\section{Examples and applications}

We continue this section by presenting an example of a nonscalar classical Bessel matrix weight with support on \( \gamma \),
%\begin{exa}
%Let consider the weight matrix
\begin{align*}
W(x) = \begin{bmatrix}
 (1 + c^2 ) x^2 & c \\
 c & 1
\end{bmatrix}
x^{a-2} \Exp{-\frac{b}{x}},
\end{align*}
\( b \in \mathbb C : \Re (b) > 0 \),
\( a \in \mathbb R \setminus \{ -1, -2, \dots \}\), 
\(c \in \mathbb R \). 
The sequence of matrix orthogonal polynomials, \( \{ \mathbb B_n \}_{n \in \mathbb N} \), associated with this weight is given~by
\begin{align*}
\mathbb B_n (x) = a_n \, B_{n} (x) +b_n \, B_{n-1} (x) +c_n \, B_{n-2} (x) , && n \in \mathbb N ,
\end{align*}
where \( B_n (x) = \frac{b^n}{(a+n-1)_n} y_n (x,a+2,b) \), with
\( \big\{y_n \big\}_{n \in \mathbb N} \) the generalized monic Bessel polynomials defined in~\eqref{eq:scalarbessel},~and for all \( n \in \mathbb N \),
\begin{align*}
a_n &=
\begin{bsmallmatrix}
 1 & \frac{b c}{a+2 n} \\
 0 & a (a+1) (a+n-1)-b^2 c^2 n
\end{bsmallmatrix},
 &&
b_n =
\begin{bsmallmatrix}
 0 & \frac{b^2 c n}{(a+2 n-1) (a+2 n)^2} \\
 b^2 c n & \frac{a b n \left(2 a (a+n)+b^2 c^2+2 n-2\right)}{(a+2 n-2) (a+2 n)}
\end{bsmallmatrix},
 &&
c_n =
\begin{bmatrix}
 0 & 0 \\
 0 & \alpha_n 
\end{bmatrix} ,
\end{align*}
with \( \alpha_n = \frac{b^2 (n-1) n (a+n-1) \left(a^2+a+b^2 c^2\right)}{(a+2 n-3) (a+2 n-2)^2 (a+2 n-1)} \).
%\end{exa}
It is easy to see that the matrix weight admit the 
%following
decomposition \( W (x) = W_\mathsf L (x) W_\mathsf R (x) \),~where
\begin{align*}
W_\mathsf L (x) =
\begin{bmatrix}
 x & c x \\
 0 & 1
\end{bmatrix} x^{\frac{a-2} 2}\Exp{-\frac{b}{2x}} && \text{and} &&
W_\mathsf R (x) = W_\mathsf L^\top (x),
\end{align*}
with
%\begin{align*}
\(
x^2 W_\mathsf L^\prime (x) 
 =
( A^\mathsf L x + B^\mathsf L ) W_\mathsf L (x) 
\),
%&&
% \text{
where %} &&
\(
A^\mathsf L =
\begin{bsmallmatrix}
 \frac{a}{2} & 0 \\
 0 & \frac{a-2}{2}
\end{bsmallmatrix}
\),
% &&
\(
B^\mathsf L =
\frac{b}{2} \operatorname I \).
%\end{align*}
Denoting the monic matrix orthogonal polynomials by
\( \widetilde{\mathbb B}_n \)
we see that
\begin{align*}
\widetilde{\mathbb B}_n (x)
 = \operatorname I B_n(x) 
 + \widetilde b_n B_{n-1} (x)
 + \widetilde c_n B_{n-2} (x), && n \in \mathbb N ,
\end{align*}
where
\( \widetilde b_n = a_n^{-1} b_n \)
and
\( \widetilde c_n = a_n^{-1} c_n \).
Taking into account \cite{zbMATH03048838} we can write
\begin{align}
\label{eq:recbesselescalar}
x B_n & = B_{n+1} + \beta_n B_n + \gamma_n B_{n-1} , \\
\label{eq:estruturabesselescalar}
x^2 B_n^\prime & = n B_{n+1} + g_n B_{n} + h_n B_{n-1} , && n \in \big\{ 1,2, \ldots \big\} ,
\end{align}
where
\begin{align}
\label{eq:betagamma}
\beta_n & =- \frac{a b}{(a+2 n) (a+2 n+2)} ,
 &&
 \gamma_n = -\frac{b^2 n (a+n)}{(a+2 n-1) (a+2 n)^2 (a+2 n+1)} , \\
 \label{eq:estruturabesselkrallfrink}
 g_n & = -\frac{2 b n (a+n+1)}{(a+2 n) (a+2 n+2)} ,&& h_n =\frac{b^2 n (a+n) (a+n+1)}{(a+2 n-1) (a+2 n)^2 (a+2 n+1)} .
\end{align}
From these we conclude that \( \big\{ \widetilde{\mathbb B}_n (x) \big\}_{n \in \mathbb N} \) satisfies
\begin{align*}
x \widetilde{\mathbb B}_n (x) = \widetilde{\mathbb B}_{n+1} (x)
+ \xi_n \widetilde{\mathbb B}_n (x) + \eta_n \widetilde{\mathbb B}_{n-1} (x) , && n \in \big\{ 1,2, \ldots \big\} ,
\end{align*}
with
%\begin{align*}
\(
\xi_n = \beta_n \operatorname I + \widetilde b_n - \widetilde b_{n+1} \)
and
% &&
\(
\eta_n = \gamma_n \operatorname I + \beta_{n-1} \widetilde b_n+ \widetilde c_n - \widetilde c_{n+1}-\xi_n\widetilde b_n 
\).
%\end{align*}
Furthermore,
\begin{align*}
x^2 \widetilde{\mathbb B}_n^\prime
+ \widetilde{\mathbb B}_n 
\big( A^\mathsf L x + B^\mathsf L \big)
%x^2 W_{\mathsf L}^\prime W_{\mathsf L}^{-1}
 = r_n \widetilde{\mathbb B}_{n+1} + s_n \widetilde{\mathbb B}_{n} + t_n \widetilde{\mathbb B}_{n-1} ,
\end{align*}
with
\begin{align*}
r_n & = n \operatorname I + A^\mathsf L,
 \\
s_n & = \Big( \frac b 2 + g_n \Big) \operatorname I + (n-1) \widetilde b_n + (\beta_n \operatorname I + \widetilde b_n) A^\mathsf L - r_n \widetilde b_{n+1}, \\
t_n & = h_n \operatorname I + \Big( g_{n-1} + \frac b 2 \Big) \widetilde b_n+ (n-2) \widetilde c_n + (\gamma_n \operatorname I + \beta_{n-1} \widetilde b_n+\widetilde c_n) A^\mathsf L - r_n \widetilde c_{n +1} .
\end{align*}
The last expression enables us to write
\begin{align*}
\begin{cases}
\displaystyle
x^2 \widetilde{\mathbb B}_n^\prime
+ \widetilde{\mathbb B}_n 
\big( A^\mathsf L x + B^\mathsf L \big)
 = M_n^{11} \widetilde{\mathbb B}_n + M_n^{12} ( - C_{n-1} \widetilde{\mathbb B}_{n-1}) , \\
 \displaystyle
 x^2 (-C_{n-1} \widetilde{\mathbb B}_{n-1})^\prime
+(- C_{n-1} \widetilde{\mathbb B}_{n-1} )
\big( A^\mathsf L x + B^\mathsf L \big)
 = M_n^{21} \widetilde{\mathbb B}_n + M_n^{22} ( - C_{n-1} \widetilde{\mathbb B}_{n-1} ),
 \end{cases} 
\end{align*}
where
\begin{align*}
M_n^{11} & = r_n (x \operatorname I - \xi_n) + s_n , 
 &&
M_n^{12} = (r_n \eta_n - t_n)C_{n-1}^{-1} ,
 \\
M_n^{21} & = - C_{n-1} \big(s_{n-1} - t_{n-1} \eta_{n-1}^{-1} \big) , 
 && 
M_n^{22} = C_{n-1} \big( s_{n-1} + t_{n-1} \eta_{n-1}^{-1} (x \operatorname I - \xi_{n-1}) ) C_{n-1}^{-1} ,
\end{align*}
with
%\begin{align*}
\(
C_{n}^{-1} = (-1)^n a_n^{-1}
 \begin{bsmallmatrix}
 \frac{a (a+1) n! (a+n) b^{2 n}-c^2 (n+1)! b^{2 n+2}}{\prod_{j=0}^{2 n} (a-j+2 n) \prod_{j=0}^{n+1} (a+j+n)} & 0 \\
 0 & \frac{n! b^{2 n} \left(a (a+1) (a+n-1)-b^2 c^2 n\right)}{\prod_{j=2}^{2 n-2} (a-j+2 n) \prod_{j=0}^{n-1} (a+j+n)} 
\end{bsmallmatrix} 
 a_n^{-\top}
 \).
%\end{align*}
Applying Proposition~\ref{prop:mn_Bessel} we get that
\begin{align*}
x^2 Y_n^\prime
+ Y_n 
\begin{bmatrix}
A^\mathsf L x + B^\mathsf L & \0 \\
\0 & - (A^\mathsf L)^{\top} x - (B^\mathsf L)^{\top}
\end{bmatrix}
= \widetilde M_n Y_n , && n \in \mathbb N ,
\end{align*}
where the matrix 
\( \widetilde M_n
=
\begin{bsmallmatrix} M_n^{11} & M_n^{12} \\ M_n^{21} & M_n^{22} \end{bsmallmatrix}
 \) 
%given by
%\begin{align*}
%\widetilde M_n = \begin{bmatrix} M_n^{11} & M_n^{12} \\ M_n^{21} & M_n^{22} \end{bmatrix} , && n \in \mathbb N,
%\end{align*}
%We can see that \( \big\{ \widetilde M_n \big\} \)
satisfies the zero curvature formula,
\begin{align*}
x^2 \big( T^{\mathsf L}_n \big)^\prime (x) & 
= \widetilde M^{\mathsf L}_{n+1} (x) T^{\mathsf L}_n(x) - T^{\mathsf L}_n (x) \widetilde M^{\mathsf L}_{n} (x),
 && n \in \mathbb N .
\end{align*}

\section{Application for semiclassical weights of Bessel type} \label{sec:4}

Let \( W \) be a regular Bessel matrix weight 
that admits a
%with 
factorization~\( W (z) \)
\linebreak
\( = W_{\mathsf L} (z) W_{\mathsf R} (z)\), where \( W_{\mathsf L}\) and \( W_{\mathsf R}\) satisfies~\eqref{eq:partial_Bessel_Pearson1} and~\eqref{eq:partial_Bessel_Pearson2}
such~that
\begin{align*}
z^2 W_\mathsf L^{\prime}(z) = ( h_0^{\mathsf L} +
h_1^{\mathsf L} z + h_2^{\mathsf L} z^2) W_{\mathsf L}(z),
&&
z^2 W_\mathsf R^{\prime}(z)
 = W_{\mathsf R}(z) (h_0^{\mathsf R} +
h_1^{\mathsf R} z + h_2^{\mathsf R} z^2) .
\end{align*} 
then applying Theorem~\ref{teo:mn_Bessel} we get that the matrix \( \widetilde{M}_n=z^2 M_n^{\mathsf L} \) is given explicitly by
\begin{align*}
%\begin{cases}
& (\widetilde{M}_n^{\mathsf L})_{11} = C_{n}^{-1} h_2^{\mathsf R} C_{n-1} +
(h_0^{\mathsf L} + h_1^{\mathsf L} z + h_2^{\mathsf L} z^2) +\big[ p_{\mathsf L,n}^{1}, h_1^{\mathsf L}\big] + z \Big( n \operatorname I +\big[ p_{\mathsf L,n}^{1}, h_2^{\mathsf L}\big]\Big) 
\\
& \hspace{7.5cm}
+ \big[ p_{\mathsf L,n}^{2},h_2^{\mathsf L}\big]
- p_{\mathsf L,n}^{1} h_2^{\mathsf L} p_{\mathsf L,n}^{1}-p_{\mathsf L,n}^1 , 
\\
&(\widetilde{M}_n^{\mathsf L})_{12} = (h_1^{\mathsf L} + h_2^{\mathsf L} z -
h_2^{\mathsf L} p_{\mathsf L,n+1}^{1} + p_{\mathsf L,n}^{1} h_2^{\mathsf L})
C_n^{-1} 
\\
& \hspace{4cm}
+ C_n^{-1} (h_1^{\mathsf R} + h_2^{\mathsf R} z + h_2^{\mathsf R}
p_{\mathsf R,n}^{1} - p_{\mathsf R,n+1}^{1} h_2^{\mathsf R})+(2n+1)C_{n}^{-1},
\\
& (\widetilde{M}_n^{\mathsf L})_{21} = -C_{n-1} (h_1^{\mathsf L} + h_2^{\mathsf L} z -
h_2^{\mathsf L} p_{\mathsf L,n}^{1} + p_{\mathsf L,n-1}^{1} h_2^{\mathsf L}) \\
& \hspace{4cm}
- (h_1^{\mathsf R} + h_2^{\mathsf R} z + h_2^{\mathsf R} p_{\mathsf
R,n-1}^{1} - p_{\mathsf R,n}^{1} h_2^{\mathsf R}) C_{n-1} -(2n-1)C_{n-1} ,
\\
 & (\widetilde{M}_n^{\mathsf L})_{22} = -C_{n-1} h_2^{\mathsf L} C_n^{-1} -
(h_0^{\mathsf R} + h_1^{\mathsf R} z + h_2^{\mathsf R} z^2) +\big[ p_{\mathsf R,n}^{1}, h_1^{\mathsf R}\big] - z
\Big(n \operatorname I +\big[ h_2^{\mathsf R}, p_{\mathsf R,n}^{1}\big]\Big)
 \\
& \hspace{7.5cm}
+\big[ p_{\mathsf R,n}^{2},h_2^{\mathsf R}\big]
+p_{\mathsf R,n}^{1} h_2^{\mathsf R} p_{\mathsf R,n}^{1} + p_{\mathsf R,{n}}^1 .
%\end{cases}
\end{align*}
Using the three term recurrence relation for \( \{ P_n^{\mathsf L} \}_{n \in \mathbb N}\) we get that 
\( p_{\mathsf L,n}^1 - p_{\mathsf L,n+1}^1 = \xi_n^{\mathsf L}\) and \( p_{\mathsf L,n}^2 - p_{\mathsf L,n+1}^2 = \xi_n^{\mathsf L} p_{\mathsf L,n}^1 + \eta_n^{\mathsf L}\)
where \( \gamma_n^{\mathsf L} = C_{n}^{-1}C_{n-1} \).
Consequently,
\begin{align*}
p_{\mathsf L,n}^1 &= - \sum_{k=0}^{n-1} \xi_k^{\mathsf L} , &
p_{\mathsf L,n}^2 &=
\sum_{i,j=0}^{n-1} \xi_i^{\mathsf L} \xi_j^{\mathsf L} - \sum_{k=0}^{n-1}\eta_k^{\mathsf L}.
\end{align*}
In the same manner, from the three term recurrence relation for \( \{ Q_n^{\mathsf L} \}_{n \in \mathbb N}\) we deduce that
\( q_{\mathsf L,n}^1 - q_{\mathsf L,n-1}^1 = \xi_n^{\mathsf R}
%\coloneqq
 = C_n \xi_n^{\mathsf L} C_n^{-1}\)
and \( q_{\mathsf L,n}^2 - q_{\mathsf L,n-1}^2 = \xi_n^{\mathsf R} q_{\mathsf L,n}^1 + \eta_n^{\mathsf R}\),
where \( \eta_n^{\mathsf R} = C_n C_{n+1}^{-1}\).

Now, we consider that \( W = W^\mathsf L \) and \( W^\mathsf R = \operatorname I \), and then 
use the representation for \( \{ P_n^{\mathsf L} \}_{n \in \N}\) and \( \{ Q_n^{\mathsf L} \}_{n \in \N}\) in \( z \) powers, the \( (1,2)\) and \( (2,2)\) entries read
\begin{align*}
&
\begin{multlined}[t][.95\textwidth]
 (2n+1) \xi_{n}^{\mathsf L} + h_0^{\mathsf L} 
+ h_2^{\mathsf L} ( \eta_{n+1}^{\mathsf L} + \eta_{n}^{\mathsf L} + (\xi_{n}^{\mathsf L})^2) + h_1^{\mathsf L} \xi_{n}^{\mathsf L} 
 \\
= [p_{\mathsf L,n}^1, h_2^{\mathsf L}] p_{\mathsf L,n+1}^1
- [p_{\mathsf L,n}^2, h_2^{\mathsf L}] 
- [p_{\mathsf L,n}^1, h_1^{\mathsf L}]+p_{\mathsf L,n}^1+C_{n}^{-1}p_{\mathsf L,n+1}^1C_{n}, 
\end{multlined}
 \\
&
\begin{multlined}[t][.95\textwidth]
 \big(\xi_n^{\mathsf L} \big)^2 = \eta_{n}^{\mathsf L} \big(2n-1 +h_2^{\mathsf L}(\xi_n^{\mathsf L} + \xi_{n-1}^{\mathsf L}) + [p_{\mathsf L,n-1}^1, h_2^{\mathsf L}] + h_1^{\mathsf L} \big) 
\\ 
 -
\big(2n+3 + h_2^{\mathsf L}(\xi_{n+1}^{\mathsf L}+\xi_n^{\mathsf L} ) + [p_{\mathsf L,n}^1, h_2^{\mathsf L}] + h_1^{\mathsf L} \big) \eta_{n+1}^{\mathsf L} . 
\end{multlined}
\end{align*}
Based on the above findings, we arrive to the equations
\begin{align}
\label{eq:dPIV_with_B_1}
 & 
\begin{multlined}[t][.85\textwidth]
(2n+1) \xi_{n}^{\mathsf L} + h_0^{\mathsf L} 
+ h_2^{\mathsf L} (\eta_{n+1}^{\mathsf L} + \eta_{n}^{\mathsf L} )
+\big(h_2^{\mathsf L} \xi_{n}^{\mathsf L}+ h_1^{\mathsf L}\big) \xi_{n}^{\mathsf L} -\sum_{k=0}^{n-1} \xi_k^{\mathsf L}
 -C_{n}^{-1}\sum_{k=0}^{n} \xi_k^{\mathsf L}C_{n}
 \\ = 
 \Big[ \sum_{k=0}^{n-1} \xi_k^{\mathsf L}, h_2^{\mathsf L}\Big] \sum_{k=0}^{n} \xi_k^{\mathsf L}
- \Big[\sum_{i,j=0}^{n-1} \xi_i^{\mathsf L} \xi_j^{\mathsf L} - \sum_{k=0}^{n-1}\eta_k^{\mathsf L}
, h_2^{\mathsf L}\Big] - \Big[ \sum_{k=0}^{n-1} \xi_k^{\mathsf L}, h_1^{\mathsf L}\Big],
\end{multlined}
 \\
\label{eq:dPIV_with_B_2}
 &
\begin{multlined}[t][.85\textwidth]
(\xi_n^{\mathsf L})^2-\eta_{n}^{\mathsf L} \big((2n-1)\operatorname I + h_1^{\mathsf L} + h_2^{\mathsf L}(\xi_n^{\mathsf L} + \xi_{n-1}^{\mathsf L}) \big) + \big( (2n+3) \operatorname I + h_1^{\mathsf L} \\
+h_2^{\mathsf L}(\xi_n^{\mathsf L} + \xi_{n+1}^{\mathsf L}) \big) \eta_{n+1}^{\mathsf L} 
= \eta_{n}^{\mathsf L} \Big[ \sum_{k=0}^{n-2} \xi_k^{\mathsf L} , h_2^{\mathsf L}\Big] -
\Big[\sum_{k=0}^{n-1} \xi_k^{\mathsf L} , h_2^{\mathsf L}\Big] \eta_{n+1}^{\mathsf L} .
\end{multlined}
\end{align}

We will show now that this system contains a noncommutative version of an instance of discrete Painlev\'e~IV equation.
We see, on the \emph{r.h.s.} of the nonlinear discrete equations~\eqref{eq:dPIV_with_B_1} and~\eqref{eq:dPIV_with_B_2} nonlocal terms (sums) in the recursion coefficients \( \beta_n^{\mathsf L}\) and~\( \gamma_n^{\mathsf L}\), all of them inside commutators. 
These nonlocal terms vanish whenever the three matrices \( \{ h_0^{\mathsf L}, h_1^{\mathsf L}, h_2^{\mathsf L}\}\) conform an Abelian set, so that \( \{ h_0^{\mathsf L}, h_1^{\mathsf L}, h_2^{\mathsf L},\beta_{n}^{\mathsf L}, \gamma_{n}^{\mathsf L}\}\) is also an Abelian set. In this commutative setting we~have 
\begin{align*}
&
\big(h_2^{\mathsf L} \xi_{n}^{\mathsf L}+ h_1^{\mathsf L}+(2n+1) \operatorname I \big) \xi_{n}^{\mathsf L}+ h_0^{\mathsf L} + h_2^{\mathsf L} \big(\eta_{n+1}^{\mathsf L} + \eta_{n}^{\mathsf L} \big)-p_{\mathsf L,n}^1 - p_{\mathsf L,n+1}^1
=\0
,
 \\
&
(\xi_n^{\mathsf L})^2+ \big( h_2^{\mathsf L}(\xi_n^{\mathsf L} + \xi_{n+1}^{\mathsf L}) 
 + h_1^{\mathsf L}+(2n+3)\operatorname I \big) \eta_{n+1}^{\mathsf L} -\eta_{n}^{\mathsf L} \big( h_2^{\mathsf L}(\xi_n^{\mathsf L} + \xi_{n-1}^{\mathsf L}) + h_1^{\mathsf L} +(2n-1)\operatorname I\big) =\0.
\end{align*}
In terms of
\begin{align*}
\nu_n \coloneqq \frac{ h_0^{\mathsf L} }{2}+h_2^{\mathsf L} \eta_n - p_{\mathsf L,n}^1
&&
\text{and}
&&
\mu_n \coloneqq h_2^{\mathsf L} \xi_n^{\mathsf L} +h_1^{\mathsf L}+(2n+1) \operatorname I ,
\end{align*}
the above equations reads as
\begin{align*}
-\mu_n \xi^{\mathsf L}_n & = \nu_n + \nu_{n+1} 
&& \text{and} &&
\xi^{\mathsf L}_n (\nu_n -\nu_{n+1}) = 
 \eta_{n+1} \mu_{n+1} -\eta_n \mu_{n-1} .
\end{align*}
Now, we multiply the second equation by
\( 
%\displaystyle
\mu_n\) 
and taking into account the first one we arrive~to
%\begin{align*}
\(
- (\nu_n + \nu_{n+1})(\nu_n - \nu_{n+1})
=\eta_{n+1} \mu_{n} \mu_{n+1} -\eta_n \mu_{n-1} \mu_{n} \),
%\end{align*}
and so
%\begin{align*}
\(
 \nu_{n+1}^2 - \nu_n^2
= \eta_{n+1} \mu_{n} \mu_{n+1} -\gamma_n \mu_{n} \mu_{n-1} 
\).
%\end{align*}
Hence,
\begin{align*}
\nu_{n+1}^2 - \nu_0^2 = \eta_{n+1} \mu_{n} \mu_{n+1} 
&& \text{and} &&
\xi^{\mathsf L}_n \mu_n& = - (\nu_n + \nu_{n+1})
\end{align*}
coincide to the ones 
presented in~\cite{zbMATH05692525}
(see also~\cite{zbMATH07757742})
as discrete Painlev\'e~IV
%(dPIV)
equation.
In fact, taking \( x_n = \mu_n^{-1}\) we finally arrive to
\begin{align*}
 x_{n} x_{n+1} 
 & = \frac{\big( h_2^{\mathsf L}\big)^{-1}\big(\nu_{n+1} - h_0^{\mathsf L}/2 +p_{\mathsf L,n+1}^1\big)}{\nu_{n+1}^2 - \nu_0^2} , 
 &&
\nu_n + \nu_{n+1} = \frac{\big( h_2^{\mathsf L}\big)^{-1}}{x_n}
\Big(h_1^{\mathsf L} +(2n+1)- \frac{1}{x_n} \Big) .
\end{align*} 

Now, we are able to state an non-Abelian extension of the d-PIV,

\begin{teo}
%[Non-Abelian extension of the dPIV]
Equations~\eqref{eq:dPIV_with_B_1} and \eqref{eq:dPIV_with_B_2} defines
a nonlocal nonlinear non-Abelian system for the recursion coefficients.
\end{teo}

We end this section by considering an example of semiclassical matrix Bessel weight with support on \( \gamma \),
\begin{align*}
W (x)=
\begin{bmatrix}
 1+c^2x^2 & cx\\
cx &1
\end{bmatrix}
x^{a}\Exp{-\frac b x}
%= W_{\mathsf L} \, W_{\mathsf R}
 ,
 \end{align*}
\(
b \in \mathbb C : \Re (b) > 0 \),
\( a \in \mathbb R \setminus \{ -1, -2, \dots \} \),
\( c \in \mathbb R \),
where 
\begin{align*}
W_\mathsf L (x) =
\begin{bmatrix}
 1 & c x \\
 0 & 1
\end{bmatrix} 
x^{\frac a 2} \Exp{-\frac{b}{2x}} && \text{and} &&
W_\mathsf R (x) = W_\mathsf L^\top (x),
\end{align*}
with
%\begin{align*}
\(
x^2 W_\mathsf L^\prime (x) 
 =
( h_2 x^2 + h_1 x +h_0 ) W_\mathsf L (x) 
\),
%\end{align*}
 where
%\begin{align*}
\(
h_2 =
\begin{bsmallmatrix}
 0 & c \\
 0 & 0
\end{bsmallmatrix}\),
% &&
\(
h_1 =
\frac{a}{2} \operatorname I \),
% &&
and
%&&
\(
h_0 =
\frac{b}{2} \operatorname I \).
%\end{align*}
The~sequence of matrix orthogonal polynomials, \( \{ \mathbb B_n \}_{n \in \mathbb N} \), associated with this weight is given~by
\begin{align*}
\mathbb B_n (x) = a_n \, B_{n} (x) +b_n \, B_{n-1} (x) +c_n \, B_{n-2} (x) , && n \in \mathbb N ,
\end{align*}
where \( B_n (x) = \frac{b^n}{(a+n-1)_n} y_n (x,a+2,b) \), with
\( \big\{y_n \big\}_{n \in \mathbb N} \) the generalized monic Bessel polynomials defined in~\eqref{eq:scalarbessel},~and for all \( n \in \mathbb N \),
\begin{align*}
a_n &=
\begin{bmatrix}
 1 & -c \beta_n \\
 0 & 1+ c^2 \gamma_n
\end{bmatrix},
 &&
b_n =
\begin{bmatrix}
 0 & -c \gamma_n \\
 -c \gamma_n & c^2 \beta_{n - 1} \gamma_{n}
\end{bmatrix},
 &&
c_n =
\begin{bmatrix}
 0 & 0 \\
 0 & c^2 \gamma_{n - 1} \gamma_n 
\end{bmatrix} ,
%&& n \in \mathbb N ,
\end{align*}
with \( \beta_n \) and \( \gamma_n \) given in \eqref{eq:betagamma}.
Denoting the monic matrix orthogonal polynomials by~\( \widetilde{\mathbb B}_n \)
we see that
\begin{align*}
\widetilde{\mathbb B}_n (x)
 = \operatorname I B_n(x) 
 + \widetilde b_n B_{n-1} (x)
 + \widetilde c_n B_{n-2} (x), && n \in \mathbb N ,
\end{align*}
where
\( \widetilde b_n = a_n^{-1} b_n \)
and
\( \widetilde c_n = a_n^{-1} c_n \).

We can see from~\eqref{eq:recbesselescalar}
that \( \big\{ \widetilde{\mathbb B}_n (x) \big\}_{n \in \mathbb N} \) satisfies
\begin{align*}
x \widetilde{\mathbb B}_n (x) = \widetilde{\mathbb B}_{n+1} (x)
+ \xi_n \widetilde{\mathbb B}_n (x) + \eta_n \widetilde{\mathbb B}_{n-1} (x) , && n \in \big\{ 1,2, \ldots \big\} ,
\end{align*}
with
%\begin{align*}
\(
\xi_n = \beta_n \operatorname I + \widetilde b_n - \widetilde b_{n+1}
\),
% &&
\(
\eta_n = \gamma_n \operatorname I + \beta_{n-1} \widetilde b_n+ \widetilde c_n - \widetilde c_{n+1}-\xi_n\widetilde b_n 
\).
%\end{align*}
Tacking into account~\eqref{eq:estruturabesselescalar} we arrive to
\begin{align*}
x^2 \widetilde{\mathbb B}_n^\prime
+ \widetilde{\mathbb B}_n 
\big( h_2 x^2 + h_1 x +h_0 \big)
 = h_2 \widetilde{\mathbb B}_{n+1} + r_n \widetilde{\mathbb B}_{n+1} + s_n \widetilde{\mathbb B}_{n} + t_n \widetilde{\mathbb B}_{n-1} + u_n \widetilde{\mathbb B}_{n-1} ,
\end{align*}
with
\begin{align*}
r_n & = \Big( n + \frac a 2 \Big) \operatorname I + \big( (\beta_{n+1} + \beta_{n} ) \operatorname I + \widetilde b_n \big) h_2 - h_2 \widetilde b_{n+2} ,
 \\
s_n & = \Big( \frac{a}{2} +(n-1) \operatorname I \Big) \widetilde b_{n} + \Big( \frac{b}{2}+ \frac a 2 \beta_n + g_n \Big) \operatorname I + \big( (\beta_n^2+\gamma_{n+1}+\gamma_{n} ) \operatorname I +(\beta_{n-1}+\beta_n) \widetilde b_{n} \big) h_2
 \\
 & \hspace{8.5cm}
 - h_2 \widetilde c_{n+2} - r_n \widetilde b_{n+1} , \\
t_n & = \frac{a}{2} ( \gamma_n \operatorname I +\beta_{n-1} \widetilde b_{n} +\widetilde c_{n} )+h_n \operatorname I+
\big( (\beta_{n-1} +\beta_n ) \gamma_n \operatorname I + (\beta_{n-1}^2+\gamma_{n-1}+\gamma_n ) \widetilde b_{n} \big)h_2 \\
&\hspace{5cm}
 +g_{n-1}\widetilde b_{n} +\frac{b}{2} \widetilde b_{n} +(n-2) \widetilde c_{n}
 -r_n \widetilde c_{n+1} -s_n\widetilde b_{n} , \\
u_n & = ((\beta_{n-2}+\beta_{n-1} ) \gamma_{n-1} \widetilde b_{n} + \gamma_n \gamma_{n-1}) h_2+\frac{a}{2} (\beta_{n-2} \widetilde c_{n} +\gamma_{n-1} \widetilde b_{n} ) 
 +g_{n-2} \widetilde c_{n} +h_{n-1} \widetilde b_{n} \\
& \hspace{8.5cm} +\frac{b}{2} \widetilde c_{n} - s_n \widetilde c_{n} - t_n \widetilde b_{n-1}
 ,
\end{align*}
with the \( g_n \), \( h_n \) given in~\eqref{eq:estruturabesselkrallfrink}.
The last expression enables us to express
\begin{align*}
\begin{cases}
\displaystyle
x^2 \widetilde{\mathbb B}_n^\prime
+ \widetilde{\mathbb B}_n 
\big( h_2 x^2 + h_1 x +h_0 \big)
 = M_n^{11} \widetilde{\mathbb B}_n + M_n^{12} ( - C_{n-1} \widetilde{\mathbb B}_{n-1}) , \\
\displaystyle
x^2 ( - C_{n-1} \widetilde{\mathbb B}_{n-1})^\prime
+ ( - C_{n-1} \widetilde{\mathbb B}_{n-1} )
\big( h_2 x^2 + h_1 x +h_0 \big)
 = M_n^{21} \widetilde{\mathbb B}_n + M_n^{22} ( - C_{n-1} \widetilde{\mathbb B}_{n-1}) ,
 \end{cases} 
\end{align*}
where
\begin{align*}
M_n^{11} & = h_2 \big( (x \operatorname I - \xi_{n+1})(x \operatorname I - \xi_n \big) - \eta_{n+1}) 
+ r_n (x \operatorname I - \xi_n) + s_n - u_n \eta_{n-1}^{-1} , 
 \\
M_n^{12} & = (h_2 (x \operatorname I + \xi_{n+1})\eta_{n} + r_n \eta_n - t_n - u_n \eta_{n-1}^{-1} (x \operatorname I - \xi_{n-1}))C_{n-1}^{-1} ,
 \\
M_n^{21} & = - C_{n-1} \big( h_2 (x \operatorname I - \xi_{n}) + r_{n-1} - t_{n-1} \eta_{n-1}^{-1} -u_{n-1} \eta_{n-2}^{-1} (x \operatorname I - \xi_{n-2}) \eta_{n-1}^{-1} \big) , 
 \\ 
M_n^{22} & = C_{n-1} \big( -h_2 \eta_n, + \widetilde s_{n - 1} + 
 \widetilde t_{n - 1}
 \eta_{n - 1}^{-1} (x \operatorname I - \xi_{n - 1} ) - 
 \widetilde u_{n - 1} \eta_{n - 2}^{-1} \\
 & \hspace{5cm}
 + 
 \widetilde u_{n - 1} 
 \eta_{n - 2}^{-1} (x \operatorname I - \xi_{n - 2} )
 \eta_{n - 1}^{-1} (x \operatorname I - \xi_{n - 1} ) \big) C_{n-1}^{-1} ,
\end{align*}
with
%\begin{align*}
\(
C_{n}^{-1} = a_n^{-1}
 \begin{bsmallmatrix}
1 + c^2 \gamma_{n+1} & 0 \\
 0 & 1 + c^2 \gamma_n 
\end{bsmallmatrix} 
 a_n^{-\top} \Big( \prod_{j=1}^n \gamma_j \Big)
\).
%\end{align*}
It can be seen that the matrix
\( \widetilde M_n
=
\begin{bsmallmatrix} M_n^{11} & M_n^{12} \\ M_n^{21} & M_n^{22} \end{bsmallmatrix}
 \) 
%given by
%\begin{align*}
%\widetilde M_n = \begin{bmatrix} M_n^{11} & M_n^{12} \\ M_n^{21} & M_n^{22} \end{bmatrix} , %&& n \in \mathbb N,
%\end{align*}
satisfies the zero curvature formula,
\begin{align*}
x^2 \big( T^{\mathsf L}_n \big)^\prime (x) & 
= \widetilde M^{\mathsf L}_{n+1} (x) T^{\mathsf L}_n(x) - T^{\mathsf L}_n (x) \widetilde M^{\mathsf L}_{n} (x),
 && n \in \mathbb N .
\end{align*}
Moreover, it can be checked that the \( M_n^{i j} \), \( i,j \in \{ 1,2 \}\) satisfy the analogous relations like \eqref{eq:dPIV_with_B_1} and \eqref{eq:dPIV_with_B_2}, i.e. for all \( n \in \mathbb N \),
\begin{align*}
C_{n-1}^{-1} M_n^{21} = - M_n^{12} C_{n-1}, &&
C_{n-1}^{-1} M_n^{22} = \big( M_{n-1}^{11} + M_{n-1}^{12} C_{n-1} (x \operatorname I - \xi_{n-1} ) \big) C_{n-1}^{-1} 
%, && n \in \mathbb N 
 .
\end{align*}

%\bibliographystyle{plain}
%\bibliography{reference}

\end{document}